\documentclass[11pt,leqno]{article}

\usepackage{amssymb,amsmath,amsthm,mysects,url,rotating}
 \usepackage{graphicx,epsfig}
 \usepackage{xcolor,graphicx}
\newcommand{\n}{\noindent}

\newcommand{\vp}{\varepsilon}
\newcommand{\bb}[1]{\mathbb{#1}}
\newcommand{\cl}[1]{\mathcal{#1}}

\newcommand{\ovl}{\overline}

\theoremstyle{plain}
\newtheorem{thm}{Theorem}[section]
\newtheorem{lem}[thm]{Lemma}

\newtheorem{pro}[thm]{Proposition}

\newtheorem{cor}[thm]{Corollary}

\theoremstyle{definition}

\newtheorem{dfn}[thm]{Definition}

\theoremstyle{remark}
\newtheorem{rem}[thm]{Remark}

\numberwithin{equation}{section}

\def\RR{\bb R}
\def\CC{\bb C}

\def\E{\bb E}

\def\N{\bb N}
\def\P{\bb P}

\def\T{\bb T}

\def\RR{\bb R}
\def\R{\bb R}

\def\CC{\bb C}
\def\C{\bb C}

\def\E{\bb E}

\def\N{\bb N}
\def\P{\bb P}

\def\T{\bb T}
\def\z{z}

\begin{document}
\def\d{\delta}

 \title{On  
 uniformly  bounded orthonormal Sidon systems}

\author{by\\
Gilles  Pisier\\
Texas A\&M University and UPMC-Paris VI \footnote{Partially supported by  
    ANR-2011-BS01-008-01.}}

\maketitle
\begin{abstract}
 In answer to a question raised recently by Bourgain and Lewko, we show,   that any
 uniformly  bounded  subGaussian orthonormal system
   is
  $\otimes^2$-Sidon. This sharpens their result that it is 
 ``5-fold tensor Sidon", or $\otimes^5$-Sidon in our terminology.   The proof is somewhat reminiscent of the author's original one for (Abelian) group characters, based on  ideas due to Drury and Rider.
However, we use Talagrand's majorizing measure theorem in place of Fernique's metric entropy lower bound.
We also show that  
a uniformly  bounded orthonormal system
is randomly Sidon iff it is $\otimes^4$-tensor Sidon,
or equivalently $\otimes^k$-Sidon for some (or all) $k\ge 4$.
Various generalizations
are presented, including the case of random matrices, for systems analogous
to the Peter-Weyl decomposition for compact non-Abelian groups.
In the latter setting we also include a new proof of Rider's 
unpublished result that randomly Sidon sets are Sidon, which implies that the union
of two Sidon sets is Sidon.  
 \end{abstract}
 \def\tr{{\rm tr}}

\newpage

The study of 
``thin sets" and in particular Sidon sets in discrete Abelian groups was  actively developed
in the 1970's and  1980's.
Among the early fundamental results, Drury's proof of the stability of Sidon sets under finite unions stands out (see \cite{LR}). 
  Rider's work \cite{Ri} connected Sidon sets to random Fourier
series. 
This led the author to
a characterization of Sidon sets in terms of Rudin's $\Lambda(p)$-sets
(see \cite{Pi,MaPi})
and eventually in \cite{Pis} to an arithmetic characterization of Sidon sets.
Bourgain \cite{Bo} gave a different proof, as well as a host of other results on related questions.
  The 2013 book \cite{GH} by Graham and Hare gives an account
  of this subject, updating  the  1975  one \cite{LR}  by Lopez and Ross. 
  Concerning $\Lambda(p)$-sets see Bourgain's   survey   \cite{Bo2}.
 See also \cite{LQ} for connections with Banach space theory.
  
Most of the results on lacunary sets crucially use the group structure
(\cite{Bo1} is a notable exception).
However, quite recently Bourgain and Lewko \cite{BoLe} were able
to obtain several analogues for uniformly bounded 
orthonormal systems. We pursue the same theme in this paper.

Let $\psi_2(x)=\exp x^2 -1$. 
Let $C$ be a constant.  
We will say that
an orthonormal system $(\varphi_n )$ in $L_2 (T,m)$ (here $(T,m)$ is any probability space) is a {\it $C$-subGaussian system}
 if for any sequence $y=(y_n )$ in $l_2$ we have
\begin{equation}\label{1}
 \|\sum y_n \varphi_n \|_{L_{\psi_2}}\leq C(\sum |y_n |^2 )^{\tfrac{1}{2}},
\end{equation}
where  ${L_{\psi_2}}$ is the Orlicz space  on $(T,m)$ associated to $\psi_2$,
with norm defined by
 \begin{equation}\label{orl}\|f\|_{\psi_2}  = \inf\{t> 0\mid {\bb E}\exp|f/t|^2\le e\}. \end{equation}
In \cite{BoLe}   this  is called a $\psi_2 (C)$-system but
we prefer to use a different term.
This is a variant of the notion of subGaussian random process,
as considered in \cite{Ka68} or \cite[p. 24]{MaPi}. It also appears
under the name
 ``$\sigma$-generalized Gaussian" in  e.g. \cite[p. 236]{Sto}.
 Indeed,
assuming without loss of generality that the  $\varphi_n $'s all have
 vanishing mean, then \eqref{1} holds for some $C$ iff 
 for some $\sigma>0$ we have for all $y\in \ell_2$
  \begin{equation}\label{sub0}
  \E \exp  \Re(\sum y_n \varphi_n) \le \exp{ \sigma^2 \sum |y_n|^2/2}.\end{equation}
  Clearly the latter forces $\E\varphi_n=0$ for all $n$
  while \eqref{1} does not. But this is the only significative difference.
  Indeed,   assuming  $\E\varphi_n=0$ for all $n$,
it is not hard to show that \eqref{1} implies \eqref{sub0} 
 for some $0\le \sigma<\infty $ depending only on $C$.
Conversely \eqref{sub0} for some $\sigma$
implies \eqref{1} for some $C<\infty$  depending only on $\sigma$  (and $C\simeq  \sigma$
 for the best possible values).

Classical examples of such systems include sequences of independent
identically distributed (i.i.d.) Gaussian random variables,
or independent mean zero uniformly bounded ones. They also
include Hadamard lacunary sequences of the form
$\varphi_n(t)=\exp{(iN(n)t)}$ where $\{N(n)\}$ is an increasing sequence
such that $\liminf N(n+1)/N(n)>1$. Roughly, one may interpret
\eqref{1} as expressing a certain form of independence of
the system $(\varphi_n )$. When the system is bounded in $L_\infty (T,m)$,
  the notion of Sidon system constitutes another  form of independence:
  we say that $(\varphi_n )$ is Sidon if there is a constant $\alpha $
  such that for any finitely supported scalar sequence
  \begin{equation}\label{1bis}\sum |a_n| \le \alpha  \|\sum a_n \varphi_n\|_\infty.\end{equation}
  Assume that $(\varphi_n )$  are characters on a compact Abelian group.
  Then $\eqref{1bis}\Rightarrow \eqref{1}$ by a classical result due to
  Rudin    \cite{Ru}. Conversely,
  in \cite{Pi} the author combined harmonic analysis results
  (due to Drury and Rider, see \cite{Ri})
  with probabilistic results on stationary Gaussian processes (due mainly to Fernique) to show that     $\eqref{1}\Rightarrow \eqref{1bis}$.
  Bourgain gave an alternate proof in \cite{Bo}.

  Recently, Bourgain and Lewko \cite{BoLe} considered the question whether
  the preceding implications still held
  for more general orthonormal systems bounded in $L_\infty (T,m)$.
  In such generality it is easy to see that $\eqref{1bis}\not\Rightarrow \eqref{1}$, because the direct sum of a Sidon system
  with an \emph{arbitrary} system satisfies  \eqref{1bis}. 
  Conversely Bourgain and Lewko construct an   example showing
  that also $\eqref{1}\not\Rightarrow \eqref{1bis}$, but 
  that \eqref{1} nevertheless implies
  a weak form of \eqref{1bis},  namely they show in \cite{BoLe} 
  that  \eqref{1}  implies that the 
  system $\{\varphi_n(t_1)\cdots \varphi_n(t_5)\}$ defined on the 5-fold product $T\times \cdots \times T$
  satisfies  \eqref{1bis}. Since the latter clearly implies  \eqref{1bis} 
  when the $\varphi_n$'s are group (or semi-group) morphisms, this provides one more proof
  of  $\eqref{1}\Rightarrow \eqref{1bis}$ for characters.
  Naturally they raised the question whether
  5-fold can be replaced by 2-fold, which would then be optimal.
  Our main result Theorem \ref{t1} gives a positive answer.
  We give a more general version in \S \ref{s1}
which  leads to several possibly interesting variants.
The proof makes crucial use of a consequence (see Lemma \ref{lem1})
of Talagrand's majorizing measure Theorem  from \cite{Ta}.

In \S \ref{srs} we consider the analogue of random Fourier series
for uniformly bounded 
orthonormal systems.
We call randomly Sidon the systems that satisfy the analogue of
Rider's condition (i.e. that satisfy \eqref{1bis}
with the right hand side replaced by the average over all signs
of $\|\sum\pm a_n \varphi_n\|_\infty$), and 
 we prove that 4-fold tensor Sidon
is  equivalent to randomly Sidon. Thus  the $k$-fold variant of the Sidon
property
(we   use for this the term $\otimes^k$-Sidon)
is the same notion for all $k\ge 4$.

In \S \ref{s3} we apply a similar generalization to the natural
``non-commutative" analogue of 
  Sidon sets   on non-Abelian compact groups.
  Here orthonormal functions   are replaced by matrix valued
  functions (generalizing irreducible representations),
 for  which the entries suitably renormalized
 form an orthonormal system.
We obtain  an analogue of 
subGaussian $\Rightarrow$ $\otimes^2$-Sidon
(see Corollary \ref{69'}).

 The simplest case of interest is provided by
a random $d\times d$-matrix 
$$t\mapsto [\varphi_{ij}(t)] $$  
with orthonormal entries  satisfying \eqref{1},
and such that for some $C'$  we have a uniform bound
 \begin{equation}\label{61}\sup\nolimits_{t\in T} \|d^{-1/2}\varphi(t) \|_{M_d}\le C'.\end{equation} 
Then   Corollary \ref{69'} (see also Remark \ref{71''}) shows that there is a constant $\alpha=\alpha(C,C')>0$
such that for any matrix $a\in M_d$ we have
 \begin{equation}\label{60} \alpha\tr|a| \le \sup\nolimits_{t_1,t_2\in T}|\tr(d^{-1} \varphi(t_1)\varphi(t_2) a)|.\end{equation}
The prototypical example of $\varphi$ satisfying \eqref{60} and \eqref{61} (with $C'=\alpha=1$) with orthonormal  entries satisfying
  \eqref{1} for some numerical $C$ (independent of $d$) is the case when $d^{-1/2}\varphi$ is a random unitary $d\times d$-matrix uniformly distributed over the
unitary group. In \S \ref{exam} we illustrate by an example the possible applications
   to matrices of our generalized setting. 
   
 In \S \ref{rsms} we consider the notion of ``randomly Sidon"
 for matrix valued functions. We obtain  an analogue of 
 randomly Sidon $\Rightarrow$ $\otimes^4$-Sidon
(see Theorem \ref{rs4}).
   
   In \S \ref{ssco} we briefly discuss   a reinforcement
   of the implication [$\{\varphi_{ij}\}$ $C$-subGaussian ] $\Rightarrow$  \eqref{60} valid when
   $d^{-1/2}\varphi$ is a representation $\pi$ on a compact group. In that case
   it suffices to assume that the character of $\pi$ (namely $t\mapsto \tr(\pi(t))=d^{-1/2}\sum\varphi_{ii}$)
   is $C$-subGaussian.

\section{Sidon systems}\label{s1}

\begin{thm}\label{t1}
 Let $(\varphi_n )$ be an orthonormal system satisfying \eqref{1} and moreover such that 
 \begin{equation}\label{0}\|\varphi_n \|_\infty \leq C^\prime\end{equation}
  for any $n\geq1$.  Then there is a constant $\alpha=\alpha(C,C^\prime)$ such that $\forall a\in \ell_1$
 \begin{equation}\label{2}
  \sum |a_n |\leq\alpha \sup_{(t_1,t_2)\in T\times T}|\sum  a_n \varphi_n (t_1 )\varphi_n (t_2 )|. 
 \end{equation}
\end{thm}
\begin{proof} This will be deduced from the more general Corollary \ref{cbl} below.
\end{proof}
\begin{rem} The same proof also shows  that
$$
  \sum |a_n |\leq\alpha \sup_{(t_1,t_2)\in T\times T}|\sum  a_n \varphi_n (t_1 )\ovl{\varphi_n (t_2 )}|. 
$$
Actually, assuming that $(\varphi_n^1)$ and $(\varphi_n^2)$
are two uniformly bounded orthonormal systems satisfying \eqref{1}
with respective constants $C_1,C_2$ and bounds $C'_1,C'_2$.
we will show there is  $\alpha=\alpha(C_1,C_2,C'_1,C'_2)$ such that $\forall a\in l_1$
 \begin{equation}\label{10}
  \sum |a_n |\leq\alpha \sup_{(t_1,t_2)\in T\times T}|\sum  a_n \varphi^1_n (t_1 )\varphi^2_n (t_2 )|. 
 \end{equation}
\end{rem}

In  \cite{BoLe} the system is called $\otimes^2$-Sidon
 if \eqref{2} holds, and $\otimes^k$-Sidon
  if  \eqref{2} holds with $k$-factors
($\varphi_n (t_1 )\varphi_n (t_2 )\cdots \varphi_n (t_k )$)
in place of 2. 
\noindent Theorem \ref{t1} answers the question raised in \cite{BoLe},
whether  \eqref{1} implies $\otimes^2$-Sidon.
 According to  \cite{BoLe}, \eqref{1} implies $\otimes^5$-Sidon
but not $\otimes^1$-Sidon, so ``2" is optimal.

When the $\varphi_n$'s are  Abelian group characters
Theorem \ref{t1} was established in \cite{Pi}.
 Our method closely follows  our original approach in \cite{Pi}, modulo the later progress allowed by Talagrand's majorizing measure theorem from \cite{Ta}. One could also use the subsequent
proof of
 Talagrand's Bernoulli conjecture by Bednorz and Lata\l a
 \cite{BLa}, and use Bernoulli
random variables in a similar fashion
(as we did in our initial draft), but we will content ourselves
 with  the Gaussian case.
  
  Let $(g_n )$  (resp. $(g^{\RR}_n )$)   denote an i.i.d. 
 sequence of complex (resp. real) valued  standard Gaussian random
variables on 
a probability space $(\Omega, {\cl A}, \P)$. In the complex case
  $g_n=2^{-1/2}(g'_n+ig''_n)$
with $(g'_n)$, $(g''_n)$  mutually independent and each
having the same distribution as $(g^{\RR}_n )$.

It is worthwhile to record here the following two easy and
well known observations:
for any Banach space $B$ and any $x_1,\cdots,x_N\in B$
 \begin{equation}\label{c1}
 2^{-1/2} \E\|\sum x_n g^{\RR}_n \| \le 
 \E\|\sum x_n g_n \|\le 2^{1/2} \E\|\sum x_n g^{\RR}_n \|.
 \end{equation}
For  any  matrix $a\in M_N$ with $\|a\|_{M_N}\le 1$, we have
 \begin{equation}\label{c2}
 \E\|\sum\nolimits_{i=1}^N   x_i \sum\nolimits_j a_{ij} g_j \|\le   \E\|\sum\nolimits_1^N  x_j g_j \| .
 \end{equation}
 Indeed, since equality holds when $a$ is unitary, 
 \eqref{c2} follows by an extreme point argument.
 
 In the sequel, we mostly use  the complex Gaussians $(g_n)$ but one could use 
 the real ones $(g^{\RR}_n)$, except that it introduces irrelevant factors
 equal to 2 at various places.

In addition to  Talagrand's work,
 the   crucial ingredient  in the proof is an ``interpolation property" 
 which uses the following observation (for convenience we state this only for $N<\infty$
 but this restriction is not necessary).

 \begin{lem}\label{lem0} Let $0<\d<1$ and  $N\ge 1$.
 Let $P_1$ be the orthogonal projection onto the
span of $[g_1,\cdots,g_N]$.  There is 
 an operator $T_\d:\ L_1(\P)\to L_1(\P)$  such that
 for some $w_0(\d)$ (independent of $N$) we have
  \begin{equation}\label{e7} T_\d(g_n)=g_n \ \forall n=1,\cdots,N,\quad 
 \|T_\d\|\le w_0(\d) \text{ and  }
 \|T_\d(1-P_1) :\ L_2(\P)\to L_2(\P)\|\le \d.\end{equation}
 Moreover the same result
 holds for $(g^{\RR}_n)$.
 \end{lem}
  \begin{proof}  
  First consider the real case when $(g_n)$ are real valued Gaussians.
We have a classical (hyper) contractive semigroup
(sometimes called the Mehler semigroup)
$t\mapsto T(e^{-t})$
multiplying the Hermite polynomials of (multivariate) degree $d$ by  $e^{-dt}$.
We prefer to replace $e^{-t}$ by $\d$.
More explicitly, if  we work on
  $\RR^N$ equipped with the standard Gaussian measure $\gamma$, then 
 \begin{equation}\label{ou}T(\d)(f)=\int  f(\d x+ (1-\d^2)^{1/2}y)  d\gamma(y).\end{equation}
Let $P_0$ be
the projection  onto the constant function 1. 
We  can then take $T_\d=\d^{-1} (T(\d)-P_0)$.
Then \eqref{e7} holds  with $w_0(\d)=2/\d$.\\
The complex case requires a small adjustment:
We write $g_n=2^{-1/2}(g'_n+ig''_n)$ and we apply
the preceding argument   in  $\RR^{2N}$. This
provides us with $S_\d$ such that 
$S_\d(g_n)=g_n$ but also $S_\d(\ovl{g_n})=\ovl{g_n}$
and $ \|S_\d(1-Q) :\ L_2(\P)\to L_2(\P)\|\le \d$,
where $Q$ is the orthogonal projection onto
the orthogonal of span$[g_n, \ovl{g_n}]$.
Consider now for any $z\in \T$ the measure preserving
mapping $V_z:\ L_1 \to L_1$
taking $g_n$ to $zg_n$, and let
$W=\int \bar z V_z$.  Note
$W(g_n)=g_n$ and $W(\ovl{g_n})=0$.
Clearly $\|W:\ L_p \to L_p\|\le 1$
for any $1\le p\le \infty$ and in particular for $p=1,2$.
Then the operator $T_\d= WS_\d$ satisfies
\eqref{e7}.
   \end{proof}
 We will use Talagrand's work in the form of the following result from \cite{Ta} (see also \cite[\S 2.4]{Ta2}).
\begin{lem}[\cite{Ta}]\label{lem1}
 Assume \eqref{1}.  There is a numerical constant $K$ such that for any $N$ and any $x_1 ,\cdots,x_N$ in $l_\infty$ we have
 \begin{equation}\label{3}
  \int\|\sum \varphi_n (t)x_n \|_\infty dm(t)\leq KC\int\|\sum  g_n x_n \|_\infty d\mathbb{P}.
 \end{equation}
   \end{lem}

\begin{proof}
By \eqref{c1} we may assume $(g_n)$ are real Gaussians if we wish.
 In the tradition originating in Slepian's comparison Lemma,
  \eqref{3} is  immediate from 
    Theorem 15 in \cite{Ta}. This is a simple Corollary of the Fernique majorizing measure conjecture, proved by Talagrand in \cite{Ta}.  \end{proof}

The next result, originally due to Mireille L\'evy \cite{Lev} in the case $p=1$
is a consequence of the Hahn-Banach theorem.
We use the following notation: Let $u\ :L_p (\mathbb{P})\to L_p (m)$
be a linear operator. We say that $u$ is regular
if  $\exists C$ such that for any $N$ and any $x_1 ,\cdots,x_N \in L_p (\mathbb{P})$ we have
\begin{equation}\label{21}  \| \sup_n |u(x_n)| \|_p \le C \| \sup_n |x_n| \|_p .\end{equation}
More generally, we extend this definition to any
linear operator $u\ :E\to L_p (m)$ defined only on a subspace
$E\subset  L_p (\mathbb{P})$.
We denote by
$\|u\|_{reg}$ the smallest   constant $C$
for which \eqref{21} holds for any finite set
$x_1 ,\cdots,x_N \in E$. By a well known
property of $L_1$-spaces with respect to the projective tensor product
when $p=1$ we have $\|u\|_{reg}=\|u\| $. 

\begin{pro}[\cite{Lev,Pi3}]\label{p2}
Let $(\Omega,\mathbb{P})$ and $ (T,m)$ be arbitrary measure spaces.
Let $1\le p\le  \infty$. \\
 Let $\{g_n \}\subset L_p (\Omega,\mathbb{P}),\,\{\varphi_n \}\subset L_p (T,m)$ be arbitrary sets.  The following are equivalent, for a fixed constant $C$.
 \begin{itemize}
  \item[\rm (i)] For any $N$ and any $x_1 ,\cdots,x_N \in \ell_\infty$
   \begin{equation}\label{g2} \int\|\sum \varphi_n x_n \|^p_\infty dm\leq C^p\int\|\sum  g_n x_n \|^p_\infty d\mathbb{P}.\end{equation}
   \item[\rm (i)'] Same as (i) for any Banach space $B$ any $N$ and any $x_1 ,\cdots,x_N \in B$.
  \item[\rm (ii)] There is a regular operator $u\ :L_p (\mathbb{P})\to L_p (m)$ with $\| u\|_{reg}\leq C$ such that $u(g_n )=\varphi_n$.
 \end{itemize}
\end{pro}
 
 Note that (i) $\Leftrightarrow$ (i)' holds because any separable
 Banach space embeds isometrically in $\ell_\infty$.
 
  It is worthwhile to observe that the assumption \eqref{1}
 can be replaced in Theorem \ref{t1} by the following:
 \begin{equation}\label{1nb} \text{There are a constant } C 
 \text{ and   }  u:\ L_1(\P)\to L_1(m)   \text{ with }
 \|u\|\le C \text{ such  that }\\
 \forall n \ \ u(g_n)=\varphi_n .
 \end{equation}
 
 \begin{dfn}\label{dom} 
 When \eqref{1nb} holds we will say
 that $(\varphi_n)$ is $C$-dominated by $(g_n)$.
 \end{dfn}
 Combining Proposition \ref{p2}   with Talagrand's result, we find
  \begin{thm}\label{p8}
There is a numerical constant $\tau_0$ such that 
any $C$-subGaussian sequence $\{\varphi_n\}\subset L_1(T,m)$
is $\tau_0C$-dominated by $(g_n)$. 
\end{thm}
\begin{rem} It would be interesting to find the best constant $\tau_0$.
We suspect there may be an explicit formula for the kernel of the operator
  $u$ by which a subGaussian system (say assumed satisfying \eqref{sub0} with $\sigma=1$)
  is dominated by $(g_n)$. 
  \end{rem}
  \begin{rem}[Comparing $L_{\psi_2}$ and $L_p$]\label{lamp} Let $f$ be a random variable on $(T,m)$.
  Recall that $f\in L_{\psi_2}$ (i.e. $\|f\|_{\psi_2}<\infty$)  
  iff $f\in \cap_{p<\infty}L_p$  and $\sup_{p<\infty} p^{-1/2}\|f\|_p<\infty$.
  Moreover, $f\mapsto \sup_{2\le p<\infty} p^{-1/2}\|f\|_p$ is a norm equivalent
  to the norm \eqref{orl} on $L_{\psi_2}$. These elementary and well known facts are proved using Stirling's formula and the Taylor expansion of
  the exponential function.
   \\
  Therefore, a system $(\varphi_n)$ is subGaussian (i.e. satisfies \eqref{1})
  iff there is a constant $L$ such that for any $y\in \ell_2$
  and any $2\le p<\infty$
  \begin{equation}\label{111}\|\sum y_n \varphi_n\|_p \le L p^{1/2} (\sum |y_n|^2)^{1/2}.\end{equation}
  Moreover the smallest $C$ in \eqref{1}
  and the smallest $L$ in \eqref{111}
  are equivalent quantities, up to numerical factors. \\
  In particular, a sequence of characters on a compact Abelian group
  is subGaussian iff it is a $\Lambda(p)$-set (in Rudin's sense \cite{Ru})
 for all $ 2< p<\infty$  with $\Lambda(p)$-constant   $O(p^{1/2} )$.
 See Bourgain's survey for more information on $\Lambda(p)$-sets.
  \end{rem}

   We will prove a 
 more ``abstract" form of Theorem \ref{t1}.
 Note that when dealing with
 a tensor $T\in L_1(m_1)\otimes L_1(m_2)$ 
 say $T=\sum x_j\otimes y_j$
 then the projective and injective tensor product
 norm denoted respectively by $\|\cdot\|_\wedge$ and $\|\cdot\|_\vee$
 are very simply explicitly described by
 $$\|T\|_\wedge=\int |\sum x_j (t_1) y_j (t_2)|dm_1(t_1)dm_2(t_2)$$
 $$\|T\|_\vee=\sup \{|\sum \langle x_j ,\psi_1\rangle   \langle y_j ,\psi_2\rangle |\mid \|\psi_1\|_\infty \le 1,
 \|\psi_2\|_\infty\}.$$

 \begin{thm}\label{nt1} Let $(T_1,m_1),(T_2,m_2)$
 be two probability spaces.
 Let $(g_n)$ be an i.i.d. sequence of complex Gaussian random
 variables as above.
  For any $0<\d<1$ there is 
  $  w(\d)>0$ for which the following property holds.
  Let $(\varphi^1_n)$ and  $(\varphi^2_n)$ ($1\le n\le N$) be functions
  respectively in $L_1(m_1)$ and $ L_1(m_2)$. 
  Assume that $(\varphi^1_n)$ and  $(\varphi^2_n)$  are both
  $1$-dominated by $(g_n)$ (i.e. they  satisfy 
 \eqref{1nb} or equivalently \eqref{g2} 
  with   $p=C=1$).
  Then there is a decomposition in $L_1(m_1)\otimes L_1(m_2)$
  of the form
    \begin{equation}\label{n0}\sum\nolimits_1^N \varphi^1_n \otimes \varphi^2_n=t+r\end{equation}
 satisfying
  \begin{equation}\label{n1}\|t\|_{\wedge}\le w(\d)\end{equation}
      \begin{equation}\label{n2}\|r\|_{\vee}\le \d.\end{equation}
      Moreover the same result
 holds with  $(g^{\RR}_n)$ in place of $(g_n)$.
 \end{thm}
 \begin{cor}\label{cor1} In the situation of the theorem, for any
 matrix $a\in M_N$ with $\|a\|_{M_N}\le 1$, there is a decomposition
 $$\sum\nolimits_{1\le i,j\le N} a_{ij} \varphi^1_i \otimes \varphi^2_j=t+r$$
 such that \eqref{n1} and \eqref{n2} hold.
  \end{cor}
  \begin{proof} We simply observe that by \eqref{c2}  we can replace
  $(\varphi^2_i)_{1\le i\le N}$ by the ``rotated" sequence
  $ (\sum\nolimits_j  a_{ij} \varphi^2_j)_{1\le i\le N}$, which still satisfies \eqref{1nb}.
    \end{proof}
    The assumption that $(\varphi_n)$ is orthonormal in Theorem \ref{t1}
will now  be weakened.
It suffices to assume given a system $(\psi_n)$ in $L_\infty$  that  is biorthogonal
to $(\varphi_n)$, in the sense
that $\int \varphi_n \psi_k dm=\d_{nk}$.
    The advantage of this formulation   is that it includes the Gaussian case,
 e.g. with $\varphi_n=g^{\RR}_n \|g^{\RR}_n\|_1^{-1}$ and $\psi_n={\rm sign}(g^{\RR}_n)$
 (or the complex analogue).
     \begin{cor}\label{cbl} In the situation of   Theorem \ref{nt1},
   let  $(\psi^1_n),(\psi^2_n)$  be systems biorthogonal respectively
to $(\varphi^1_n),(\varphi^2_n)$   and uniformly bounded 
respectively by $C'_1,C'_2$,
then  
\begin{equation}\label{51}
  \sum |a_n |\leq\alpha \sup_{(t_1,t_2)\in T\times T}|\sum  a_n \psi^1_n (t_1 )\psi^2_n (t_2 )|,
 \end{equation} 
 where   $\alpha$ is a constant depending only on $C'_1$ and $C'_2$.\\
 In particular, for any  
 uniformly bounded orthonormal system $(\varphi_n)$,
 subGaussian implies $\otimes^2$-Sidon.
   \end{cor}
  
      \begin{proof}
      Let $\vp_n\in \T$ be such that $|a_n|=\vp_n a_n$.
      By Corollary \ref{cor1} we have a decomposition
      $\sum \vp_n  \varphi^1_n \otimes \varphi^2_n=t+r$.
      Let $f(t_1 ,t_2)= \sum a_n \psi^1_n (t_1 )\psi^2_n (t_2 )$.
      We have
      $$\langle t+r,f\rangle=\int (\sum  \vp_n \varphi^1_n \otimes \varphi^2_n) (f)=\sum \vp_n a_n=
      \sum |a_n| .$$
      Therefore
      $$\sum |a_n| \le |\langle t,f\rangle| +  |\langle r,f\rangle|
      \le w(\d) \|f\|_\infty + \sum |a_n| |\langle r,\psi^1_n\otimes \psi^2_n\rangle|
      \le  w(\d) \|f\|_\infty +\d C_1'C_2'  \sum |a_n| .$$
Choosing $\d$ such that $\d C_1'C_2' =1/2$
we obtain the conclusion with $\alpha= 2 w(\d) $.\\
For the last part of Corollary \ref{cbl} recall that by Theorem \ref{p8} subGaussian
implies domination by $(g_n)$. Thus the last assertion
is obtained by taking (up to a renormalization, as in Remark \ref{70})
$\varphi^1_n=\varphi^2_n =\varphi_n$
and $\psi^1_n=\psi^2_n=\ovl{\varphi_n}$.
       \end{proof}
         \begin{rem} Actually the preceding proof requires
         only that $(\psi^1_n \otimes\psi^2_n)$
         be biorthogonal to $(\varphi^1_n \otimes \varphi^2_n)$.
         For instance, it suffices to have 
         $(\psi^1_n)$ biorthogonal to $(\varphi^1_n)$
         and to have $\int  \varphi^2_n\psi^2_n=1$ for all $n$.            \end{rem}
        \begin{rem} It is worthwhile to observe that
        Theorem \ref{nt1} actually reduces to the case
        when $\varphi^1_n=\varphi^2_n=g_n$. Indeed,
        once  we have obtained a decomposition
       $ \sum\nolimits_1^N g_n \otimes g_n=t+r$,
      we simply let
     $$ \sum\nolimits_1^N \varphi^1_n \otimes \varphi^2_n=(u_1\otimes u_2)(t)+(u_1\otimes u_2)(r).$$ 
        However, it turns out that the proof below
        is essentially the same in the  case $\varphi^1_n=\varphi^2_n=g_n$
        as in the general case, so we proceed without using the present observation.
        \end{rem}
   \begin{proof}[Proof of Theorem \ref{nt1}]
   By   \eqref{c1}  it suffices to prove this
   when the $(g_n)$'s are real valued.
   This is arguably irrelevant but it simplifies
   notation, allowing us to avoid complex conjugation.  
  Let
 $u_j:\ L_1(\P) \to L_1(m_j)$ with $\|u_j\|\le 1$
 such that $u_j(g^{\RR}_n)=\varphi^j_n$ for $1\le n\le N$.
 By classical results on $L_1$-spaces,  for any $\vp>0$ there is a finite rank projection $Q_j$ on $L_1(m_j)$ with $\|Q_j\|<1+\vp$ that is the identity on {\rm span}$[\varphi^j_1,\cdots,\varphi^j_N]$.
 Thus replacing $u_j$ by $Q_ju_j$ we may clearly assume
 that $u_j$ has finite rank. Thus each $u_j$
 can be identified with
 an element $\Phi^j\in L_\infty(\P)\otimes L_1(m_j)$.
 
 It is easy and well known  that, since $u_j$ has finite rank
  $$ \|\Phi^j\|_{L_\infty(\P;L_1(m_j) )}= \|u_j\|.$$
Indeed, if  $k$ is the rank of $u_j$, 
this is clear when $\Phi^j=\sum\nolimits_1^k x_q \otimes y_q$ 
with   $x_q$,$y_q$ measurable with respect to a finite
$\sigma$-subalgebra.
The general case can then be checked by a simple approximation argument by step functions.

 A fortiori $\Phi^j\in L_2(\P)\otimes L_1(m_j)$.  
 
 The fact that $\Phi^j$  represents $u_j$ is expressed
  for all $\psi^j\in  L_\infty(m_j)$
  by \begin{equation}\label{e8}u^*_j(\psi^j)= \int \Phi^j \psi^j dm_j. \end{equation}
   Let  
       \begin{equation}\label{e3}
       T=\int  \Phi^1(\omega)\otimes \Phi^2(\omega)
      d\P(\omega).\end{equation}
      Note that
    \begin{equation}\label{e5}
    \| T\|_{\wedge}=
  \| \int  \Phi^1(\omega)\otimes \Phi^2(\omega) d\P(\omega)\|_{\wedge} \le \|    \Phi^1\|_{L_\infty(\P; L_1(m_1))}  \|    \Phi^2\|_{L_\infty(\P; L_1(m_2))}
    \le \|u_1\| \|u_2\|.  \end{equation}
  We claim that
    \begin{equation}\label{e5bis}
    \| T\|_{\vee}\le \|u_1:\ L_2\to L_1\| \|u_2 :\ L_2\to L_1\|
       \le \|u_1\| \|u_2\|.  \end{equation}
       Indeed, 
       by \eqref{e8} we have
          \begin{equation}\label{f5bis}\langle T, \psi^1\otimes \psi^2\rangle=\int 
       u^*_1(\psi^1)(\omega) u^*_2(\psi^2)(\omega)  d\P(\omega)\end{equation}
       and hence
       $$|\langle T, \psi^1\otimes \psi^2\rangle| \le
       \|u^*_1(\psi^1)\|_2 \|u^*_2(\psi^2)\|_2 
        \le  \|u^*_1:\ L_\infty\to L_2\| 
        \|\psi^1\|_{L_\infty(m_1)} \|u^*_2:\ L_\infty\to L_2\|        \|\psi^2\|_{L_\infty(m_2)},
       $$
       from which \eqref{e5bis} follows.
 
       We
 will now use
 the orthogonal projection $P_1:\ L_2(\P)\to {\rm span}[g^{\RR}_n]$
 from Lemma \ref{lem0}.
 
     Let $S=\sum\nolimits_1^N \varphi^1_n \otimes \varphi^2_n$.
          We first claim that we have a decomposition (recall \eqref{e3})
$$S=   T+R$$
 such that  
 \begin{equation}\label{e5t} \| R\|_\vee \le \|u_1(I-P_1):\ L_2\to L_1\|\|u_2\|.\end{equation}
Note that since $S=(u_1\otimes u_2) (\sum g^{\RR}_n \otimes g^{\RR}_n)$ and $\sum g^{\RR}_n \otimes g^{\RR}_n$ represents $P_1$ we have
$$\langle S, \psi^1\otimes \psi^2\rangle=\int 
      \langle g^{\RR}_n, u^*_1(\psi^1)  \rangle \langle g^{\RR}_n, u^*_2(\psi^2)  \rangle d\P =  \int  P_1u^*_1(\psi^1) u^*_2(\psi^2)
      d\P,$$
      and hence by \eqref{f5bis}
      $$S= T+R$$
      where
      $$\langle R, \psi^1\otimes \psi^2\rangle=
    -   \int(I-  P_1)u^*_1(\psi^1) u^*_2(\psi^2).
    $$
Applying \eqref{e5bis} to $R$ we find 
\eqref{e5t} which proves our claim.
 
We will now use the ``interpolation property"
from Lemma \ref{lem0}.

Fix $0<\d<1$.
We will now replace $u_1$ by $u_1T_\d$.
Note that we still have  $u_1T_\d(g^{\RR}_n)=u_1 (g^{\RR}_n)=\varphi^1_n$
but in addition we now have
$$\|u_1T_\d\| \|u_2\|\le w_0(\d) \|u_1\|\|u_2\|\text{  and  } \|u_1T_\d(I-P_1):\ L_2\to L_1\| \le \d \|u_1\|.$$
Therefore, the preceding decomposition
becomes $S=t+r$
with $\|t\|_{\wedge} \le w_0(\d) \|u_1\|\|u_2\|\le w_0(\d)$
and $\|r\|_{\vee} \le \d\|u_1\|\|u_2\|\le w_0(\d) $.
The minor correction to pass from the real case to the complex one
leads to doubled constants $ w_0(\d), \d$.
   \end{proof}
    \begin{rem}\label{melo}
 Let $J:\ L_\infty(\P)\to L_1(\P)$, $J_2:\ L_\infty(\P)\to L_2(\P)$
 and   $J_1:\ L_2(\P)\to L_1(\P)$ be the natural inclusions,
 so that $J=J_1J_2$.
 A more abstract way to run the previous proof
 is to observe that the tensor $S$ corresponds to the operator
 $u_1J_1 P_1 J_2 u_2^*:\ L_\infty(m_2)\to L_1(m_1)$, which can be decomposed
 as $$u_1J_1 P_1 J_2 u_2^*=u_1J_1 T_\d J_2 u_2^* - u_1J_1T_\d (1-P_1)J_2u_2^*= u_1T_\d J  u_2^*- u_1J_1T_\d (1-P_1)J_2u_2^*.$$
 This is the operator version of the decomposition
 $S=t+r$.
 Then,  $\|t\|_\wedge$ is equal to the integral norm of $u_1T_\d J  u_2^*$
 which is clearly (since $J$ appears inside) $\le 
  \|u_1 T_\d\| \|u_2^*\|\le w_0(\d)
  \|u_1\|  \|u_2\|$
  and $\|r\|_\vee $ is 
$  \le 
\|u_1J_1T_\d (1-P_1)\| \|J_2u_2^*\|\le \d \|u_1J_1\| \|J_2u_2^*\|\le \d
\|u_1\|  \|u_2\|$.

   \end{rem}
 \begin{rem}\label{mel}
 It is known 
 that the best estimate for $w_0(\d)$ with properties
 \eqref{e7} is $w_0(\d)=O(\log(1/\d))$.
This follows easily   from 
a result proved already in the Sidon set context by J.F. M\'ela, namely
 Lemma 3 in  \cite{Me}.  
The latter says that for any $0<\d<1$
 there is a measure $\sigma$ on $[0,1]$
such that $\int s d\sigma(s) =1$,
  $|\int s^n d\sigma(s)|\le \d  $ for all odd $n$
  and $|\sigma|([0,1])\le C |\log\d|$ with $C$ independent of $\d$. 
 Now let  $T(s)$ be the operator defined in \eqref{ou}.
We have    $T(s)=\sum\n_{n\ge 0} s^n P_n$
  where the $P_n$'s are the orthogonal projections onto
  the span  (or the ``chaos") of Hermite polynomials of degree $n$.  Consider then
  $$T_\d=\int ( T(s)-T(-s))/2 \ d\sigma(s)=P_1+\sum\n_{n \text{ odd}>1} P_n \int s^n d\sigma(s).$$
  It is easy to check \eqref{e7} with $w_0(\d)\le C |\log\d|$.
  Of course this implies
 the same growth for $w(\d)$.
 
 An alternate proof can be given using complex interpolation by the same idea
 as in  \cite[p. 11]{Piexp}.     \end{rem}
   
    \begin{rem}[$\otimes^k$ implies $\otimes^{k+1}$]\label{mel2} Let $(\psi^1_n)$ be a Sidon sequence.
    Then for any sequence $(\psi^2_n)$ such that
    $\delta=\inf\|\psi^2_n\|_1>0$, the sequence
    $(\psi^1_n\otimes \psi^2_n)$ is Sidon. Indeed, for any fixed $s$
    we have
      $\sum |a_n \psi^2_n(s)| \le \alpha \|\sum a_n  \psi^1_n \psi^2_n(s)  \|_\infty$,
     which, after integration over $s$, implies 
      $\delta\sum |a_n | \le \alpha \|\sum a_n \psi^1_n\otimes \psi^2_n   \|_\infty$.
      In particular,  for a uniformly bounded othonormal system,
      $\otimes^k$-Sidon implies $\otimes^{k+1}$-Sidon.
       \end{rem}
       
       \begin{rem}[On homogeneity]\label{70}
 Assume that 
 $(\varphi_n)$ is biorthogonal to 
 $(\psi_n)$ and $\|\psi_n\|_{\infty  }\le C'$. Let $c>0$.
 If $(\varphi_n)$ is $c$-dominated by $(g_n)$,
 then $(c^{-1}\varphi_n)$ is $1$-dominated by $(g_n)$,
  is biorthogonal to  $(c\psi_n)$, and the latter
  has norm $\le cC'$ in ${L_\infty}$. Thus Corollary \ref{cbl}
  applies in this case too.
\end{rem}

\begin{cor}\label{69} 
Let $\{\psi_n\mid 1\le n\le N\} \subset L_\infty(m)$
and $\{\varphi_n\mid 1\le n\le N\} \subset L_1(m)$.
Let $a_{ij}=\langle \varphi_i, \psi_j\rangle$.
Assume that $a=[a_{ij}]$ is invertible and   $\|a^{-1}\|_{M_N}\le c$.
Then if $(\varphi_n)$ is $1$-dominated by $(g_n)$,
and if   
$ \sup\n_n\| \psi_n\|_\infty\le C'$
there is a number $\alpha=\alpha(c,C')$
(depending only on $c$ and $C'$)  
such that
$(\psi_n)$ is $\otimes^2$-Sidon with 
constant $\alpha$.  
\end{cor}
\begin{proof}
There is a system $\{\varphi'_n\mid 1\le n\le N\} \subset L_1(m)$ that is
biorthogonal to $(\psi_n)$ and $c$-dominated by $(g_n)$.
 Indeed, setting $b=a^{-1}$, and
 $\varphi'_i=\sum_k b_{ik} \varphi_k$ we have
 $\langle \varphi'_i, \psi_j\rangle =\sum_k b_{ik} a_{kj}= \d_{ij}$.
 Clearly, $(\varphi'_n)$
 is $1$-dominated by $(\sum_k b_{nk} g_k)$,
 and by the rotational invariance of Gaussian measure,
 the latter is $c$-dominated by $(g_n)$. Therefore
 the present statement follows from
 Corollary \ref{cbl} (and Remark \ref{70}).
\end{proof}
 
 \begin{rem}[On almost biorthogonal systems]\label{R71}
  In the situation of the preceding Corollary,
  let $\theta=\|a-I\|$. If $\theta<1$ then
  $\|b\|\le (1-\theta)^{-1}$.
\end{rem}

For convenience, we record here the following elementary fact.

\begin{lem}\label{lr} Let $\{g_n\mid 1\le n\le N\}\subset L_1(\P)$ and
$\{\varphi_n\mid 1\le n\le N\}\subset L_\infty(m)$.
Assume there is $T:\ L_1(\P)\to L_1(m)^{**}$ of norm $1$
such that  \begin{equation}\label{65}
\forall n,k\le N\quad \langle T(g_n), \varphi_k\rangle=\delta_{nk}.\end{equation}
Then for any $\vp>0$ there is $T^\vp:\ L_1(\P)\to L_1(m)$
with norm $\le 1+\vp$ satisfying \eqref{65}. 
\end{lem}
\begin{proof} Fix $0<\vp<1$.
Let $E={\rm span}\{T(g_n)\mid 1\le n\le N\}\subset L_1(m)^{**}$.
The space $L_1(m)^{**}$ is an abstract $L_1$-space
and more generally a $\cl L_1$-space in the sense of \cite{LiRo}.
In particular, there is a finite rank operator 
$S:\  L_1(m)^{**} \to L_1(m)^{**}$ with norm $< 1+\vp/4$
that is the identity on  $E$. Let $F$ be the range of $S$.
Note $E\subset F$.
By the local reflexivity principle, the inclusion $F\subset L_1(m)^{**}$
is the weak* limit of a net $J_i:F\to L_1(m)$
with $\|J_i\|<1+\vp$. By a simple perturbation argument
(see \cite{LiRo} for details) we may adjust $J_i$
so that $\langle J_i(e) ,\varphi_k\rangle= \langle e ,\varphi_k\rangle$
for any $e\in F$ (and hence for any $e\in E$). Then $T^\vp=J_iST:\ L_1(\P)\to L_1(m)$
satisfies \eqref{65} and $\|T^\vp\|\le 1+\vp $ .
\end{proof}


   \section{Randomly Sidon   systems}\label{srs}
   
   In this section, we denote simply by $(g_n)$ the sequence
   denoted previously by $(g^\C_n)$.
   In connection with Rider's paper \cite{Ri},
 let us  say that
   a sequence $(\psi_n)$
   is randomly Sidon if there is a constant
   $C$ such that for any finite scalar sequence $(a_n)$ we have
   $$\sum |a_n|\le C \E\|\sum g_n a_n \psi_n\|_\infty.$$
   Clearly, Sidon implies randomly Sidon.\\
   Assuming $(\psi_n)$ bounded in $L_\infty$, it is easy to see by a truncation
   argument (as in \cite{Pi} or in Lemma \ref{68} below)
   that this is equivalent  to the same property
   with a Bernoulli sequence $(\vp_n)$
   (i.e. independent uniformly distributed choices of signs) in place of $(g_n)$.
   The latter case was considered by Rider \cite{Ri}
   when the $\varphi_n$'s are
     distinct characters $(\gamma_n)$    on a compact Abelian group
     and
    he proved  that 
   randomly Sidon sets of characters are Sidon. 
      
    Assume $(\psi_n)$ bounded in $L_\infty$.
   Bourgain and Lewko \cite{BoLe} observed using Slepian's lemma
   (see Remark \ref{rbath0})
     that, for any fixed $k$,
     if the $k$-fold tensor product $(\psi_n\otimes\cdots\otimes\psi_n)$
     is randomly Sidon, then $(\psi_n)$ is randomly Sidon. Thus
      every uniformly bounded $\otimes^k$-Sidon sequence
   is randomly Sidon. In particular,
   they proved that every uniformly bounded 
  orthonormal system  satisfying \eqref{1} is 
randomly Sidon.
   
   In the remarks that follow we try to clarify the relationship
   between this notion and the notion
   of sequence dominated by $(g_n)$ introduced in Definition
   \ref{dom}.

   \begin{pro}\label{rs1}      Consider  a sequence 
 $(\varphi_n)_{1\le n\le N}$ in $L_1(m)$,
The following properties are equivalent.
\begin{itemize}
\item[(i)] For any   $(f_1,\cdots,f_N)$ in  $L_\infty(m)$ we have
$\sum |\langle f_n, \varphi_n\rangle |\le \E \|\sum g_n f_n\|_\infty.$
\item[(i)'] For any   $(f_1,\cdots,f_N)$ in  $L_\infty(m)$ we have
$|\sum \langle f_n, \varphi_n\rangle |\le \E \|\sum g_n f_n\|_\infty.$
\item[(ii)] There is an operator $u:\ L_1\to L_1$ with $\|u\|\le 1$ such that
$u(g_n) =\varphi_n$.
\end{itemize}
\end{pro}
 \begin{proof}  The equivalence $(i)\Leftrightarrow$ (i)' is obvious
 because for any $z_n\in \T$
 the right hand side is unchanged
 when we replace $(f_n)$ by $(z_nf_n)$.\\
 Assume (i). Consider the linear form
 $$ \sum g_n f_n \mapsto \sum \langle f_n, \varphi_n\rangle $$
 and extend it by Hahn-Banach to $\xi\in L^1(\P; L_\infty(m))^*$ of norm $\le 1$
  such that $\xi( g_n \otimes f_n)= \langle f_n, \varphi_n\rangle$.
   This linear form
 $\xi\in L^1(\P; L_\infty(m))^*$
 defines an operator $u:\ L^1(\P)\to L_1(m)^{**}$ with norm
  $\le 1$ such that $\langle u(g_n ),f_n\rangle=\xi( g_n \otimes f_n)= \langle f_n, \varphi_n\rangle$.
  Therefore $u(g_n )=\varphi_n$.
   Composing $u$ with the norm $1$ projection
   from $L_1(m)^{**}$ to  $L_1(m)$ (associated to the Hahn
   decomposition) we obtain (ii).\\
  Conversely, if (ii) holds
  we have
  $$\int\|\sum u(g_n)(t) \otimes f_n\|_\infty dm(t)
  \le\E \|\sum g_n f_n\|_\infty$$
  and since
  $$|\int \sum u(g_n)(t)  f_n(t)  dm(t)|\le \int {ess}\sup_s|\sum u(g_n)(t)  f_n(s)|  dm(t)$$
   and $ u(g_n) =\varphi_n$  we obtain (ii) $\Rightarrow$ (i)'.
\end{proof}

When the functions $(\varphi_n)_{1\le n\le N}$ are 
of the form  $\varphi_n= \ovl{\gamma_n}/C$, where $(\gamma_n)$ are distinct characters on a compact Abelian group,  and $C$ a fixed constant,  (i) implies\begin{itemize}
\item[(iii)] For any $a_n\in \CC$
$$|\sum a_n|\le C\E\| \sum a_n g_n\gamma_n\|_\infty.$$
\end{itemize}
By the ``sign invariance" (complex sense) of $(g_n)$ the latter is equivalent to
\begin{itemize}
\item[(iii)']  $\sum |a_n|\le  C\E\| \sum a_n g_n\gamma_n\|_\infty.$
\end{itemize}
Moreover,   a simple averaging 
   shows that
  (iii)    is equivalent to
  \begin{itemize}
\item[(iv)] For any $f_n\in L_\infty$ 
$$\sum |\hat {f}_n (\gamma_n)|\le  C\E\| \sum  g_n f_n \|_\infty.$$
\end{itemize}
Indeed, this follows from
\begin{equation}\label{rs33}\E\| \sum  g_n f_n \|_\infty\ge 
\sup_s\E\sup_t|\sum g_n \gamma_n(s) f_n(t-s)|
\ge \E\|\sum  g_n \gamma_n\ast f_n \|_\infty= 
\E\|\sum  g_n\hat {f}_n (\gamma_n) \gamma_n\|_\infty
.\end{equation}
Thus we conclude:
\begin{rem}
The set $(\gamma_n)$ is   randomly Sidon with constant $C$ iff
$( {\gamma_n})$ (or equivalently $(\ovl{\gamma_n})$) is
$C$-dominated by $(g_n)$.
\end{rem}
We now turn to
the analogous questions for more general function systems.
 \begin{pro}\label{rs2} Consider a finite sequence
$(\psi_n)_{1\le n\le N}$ in $L_\infty(m)$.
The following properties are equivalent.
\begin{itemize}
\item[(i)] For any  complex $N\times N$-matrix $[a_{nk}]$ we have
\begin{equation}\label{rs3}|\sum\nolimits_1^N a_{nn}|\le \E\|\sum\nolimits_n g_n (\sum\nolimits_k a_{nk} \psi_k)\|_\infty\end{equation}
\item[(ii)] For any $\vp>0$, there is an operator $u:\ L_1(\P)\to L_1(m)$ with $\|u\|\le 1+\vp$ such that
$(u(g_n)) $  is biorthogonal to $(\psi_n) $.
\end{itemize}
 \end{pro}
 \begin{proof} 
 Let $E={\rm span}[\psi_n]$. (i) $\Rightarrow$ (ii) is proved using Hahn-Banach
 as for  Proposition  \ref{rs1}, but a priori this leads to 
 an operator $u:\ L_1\to {E}^{*}$ with $\|u\|\le 1 $ 
 and $(u(g_n)) $   biorthogonal to $(\psi_n) $. But
 since $E$ (being finite dimensional) is weak*-closed, we may identify
 ${E}^{*}$ to   $L_1/N$ where $N$ is the preannihilator of $E$.
 Applying the lifting property of $L_1$-spaces, we obtain (ii).
 More precisely, for any $\vp>0$ there is a subspace
 $G\subset L_1(\P)$ containing $\{g_n\}$ that is $(1+\vp)$-isomorphic
 to a finite dimensional $\ell_1$-space and $(1+\vp)$-complemented
 in $L_1(\P)$. Then it suffices to lift $u_{|G}$ and that is immediate.
 The proof  of (ii) $\Rightarrow$ (i) is similar to
 the one for (ii) $\Rightarrow$ (i)' in  Proposition  \ref{rs1}.
 \end{proof}
 
   \begin{rem}\label{Rrs4} By Theorem \ref{nt1}, if  a sequence $(\psi_n)$ assumed bounded in $L_\infty$ satisfies \eqref{rs3} for all $N$
   then $(\psi_n\otimes \psi_n)$ is   Sidon (in other words 
   $(\psi_n)$ is $\otimes^2$-Sidon). \end{rem}
   \begin{rem}\label{sav1}
   In the converse direction, let $(\varphi_n)$ and $(\psi_n)$
   be mutually biorthogonal sequences in $L_\infty$.
   Assume $\|\varphi_n\|_\infty\le 1$ for all $n$.
  We claim that if $(\psi_n)$ is randomly Sidon
   with constant $\alpha$, then
   $(\psi_n\otimes \psi_n)$ satisfies \eqref{rs3} 
     with the same constant $\alpha$.
   Indeed, we have 
   $$  \E\|\sum\nolimits_n g_n (\sum\nolimits_k a_{nk} \psi_k\otimes \psi_k ) \|_\infty
   \ge \sup_s \E\|\sum\nolimits_n g_n (\sum\nolimits_k a_{nk} \psi_k(s)\otimes \psi_k(\cdot) ) \|_\infty $$
   and since $(g_n \varphi_n(s))$ is   1-dominated by $(g_n)$
   $$\ge \sup_s \E\|\sum\nolimits_n g_n \varphi_n(s)  (\sum\nolimits_k a_{nk}\psi_k(s)\otimes \psi_k(\cdot) ) \|_\infty .
   $$
   Therefore, integrating in $s$ and using Jensen again
   we obtain
   $$  \E\|\sum\nolimits_n g_n (\sum\nolimits_k a_{nk}\psi_k\otimes \psi_k ) \|_\infty
   \ge \E\|\sum\nolimits_n  a_{nn} g_n    \psi_n(\cdot)  \|_\infty 
   \ge \alpha^{-1}\sum |a_{nn}|.$$
   This proves our claim.
   \end{rem}
    \begin{rem}[randomly $\otimes^k$-Sidon implies randomly Sidon]\label{rbath0} As observed by Bourgain and Lewko
\cite{BoLe}, a well known variant (due to Sudakov) of Slepian's comparison Lemma
shows that if a  uniformly bounded system is
randomly $\otimes^k$-Sidon for some   $k\ge 2$
then it is already randomly  Sidon.  
This can be seen by an idea
due to Simone Chevet \cite{Che}:  
let $K_1,\cdots,K_k$ be compact sets,   
 $x^1_j\in C(K_1),\cdots, x^k_j\in C(K_k)$ finitely supported families, then,  
 $$
\E\|\sum g_j x^1_j \otimes\cdots\otimes x^k_j\|_{C(K_1\times\cdots\times K_k)} \le \sqrt k
\sum_{m=1}^k (\sup\n_j \prod_{q\not=m} \|x^q_j\|)\ \E\|\sum g_j x^m_j\| .$$
  Chevet's idea can be applied in a  somewhat more general context (see \cite{Che}), it gives a similar   bound
 for
$\E\|\sum\n_{i(1),\cdots,i(k)} g_{i(1),\cdots,i(k)} x^1_{i(1)} \otimes \cdots\otimes x^k_{i(k)}\|_{C(K_1\times\cdots\times K_k)},$
with $(g_{i(1),\cdots,i(k)})$ i.i.d. real valued Gaussian.\\
In any case, if we apply this to $x^1_j =\cdots= x^k_j=\psi_k$,
$B_1 =\cdots= B_k=L_\infty$ (recall $L_\infty$ is isometric to $C(K)$ for some $K$),
we find that, for $(\psi_j)$ uniformly bounded in $L_\infty$, randomly $\otimes^k$-Sidon implies randomly Sidon.
\end{rem}

 \begin{thm} \label{cbl2} 
 Let $(\psi_n)$ be bounded in $L_\infty(T,m)$,
 and such that there is another system
 $(\varphi_n)$   bounded in $L_\infty(T,m)$ such that
 $\int \psi_n \ovl{\varphi_k} dm=\d_{n,k}$.
 Then $(\psi_n)$  is randomly Sidon
 iff it is $\otimes^4$-Sidon and this holds iff
 it is $\otimes^k$-Sidon for some (or all) $k\ge 4$.
 In particular, this is valid when $(\psi_n)$
 is a uniformly bounded orthonormal system.
    \end{thm}
     \begin{proof}
     If $(\psi_n)$  is randomly Sidon, by Remarks \ref{sav1} and \ref{Rrs4}
     $(\psi_n\otimes \psi_n)$ is $\otimes^2$-Sidon, which means
     $(\psi_n)$  is $\otimes^4$-Sidon.
       Conversely, if $(\psi_n)$ is $\otimes^k$-Sidon for some   $k\ge 4$,
     then it is randomly $\otimes^k$-Sidon, and by Remark \ref{rbath0}
     it is randomly  Sidon.
\end{proof}
\begin{rem} We do not know whether, in the situation of Theorem \ref{cbl2}, 
 $\otimes^4$-Sidon already implies $\otimes^2$-Sidon.
 By \cite{BoLe}, it does not imply Sidon. 
  \end{rem}

 The next statement aims to   clarify the connection between
 the subGaussian property of a system and its domination by a Gaussian sequence.
 
We will need the following\\
{\bf Notation.} Let $(T;m)$ be a probability space.
Let  $Z $ be a scalar valued random variable on $(T;m)$.
 we denote by
   $(Z^{[k]})$   an i.i.d. sequence of copies of $Z$ on $(T;m)^{\N}$
   so that
   $$\forall t\in  T ^{\N}\quad Z^{[k]} (t)= Z(t_k).$$

We will use the following well known elementary fact.
   \\  There is an absolute constant $\theta>0$ such that
   for any $Z$ with  $\E Z=0$
 \begin{equation}\label{sub2}
 \theta^{-1}\|Z\|_{\psi_2  }
 \le \E\sup |Z^{[k]}|(\log k)^{-1/2}\le  \theta \|Z\|_{\psi_2  }. 
 \end{equation} 
 This is easily proved by relating the growth of the function
 $t\mapsto \P\{ |Z| >t\} $ to the infinite product  appearing in 
 $  \P\{ \sup |Z^{[k])}|(\log k)^{-1/2} >2\theta\}= 1 - \prod_k (1- \P\{|Z|> 2\theta (\log k)^{-1/2}\}$.
  
   \begin{pro} Let $(\varphi_n)$ ($1\le n\le N$) be a system in $L_1(T;m)$.
   Consider the following assertions, where $C$ and $C'$ are positive constants.
   \begin{itemize}
   \item[{\rm (i)}]
   The system $(\varphi _n)$ ($1\le n\le N $)
   satisfies \eqref{1} (i.e. it is $C$-subGaussian) and is such that $\E\varphi _n=0$
   for all $n$.
      \item[{\rm (ii)}] The system $(\varphi^{[k])}_n)$ ($1\le n\le N, k\in \N$)
      is $C''$-dominated by the (Gaussian i.i.d.) sequence $(g^{[k])}_n)$ ($1\le n\le N, k\in \N$).
      \end{itemize}
 Then we have $\rm (i)\Rightarrow\rm  (ii)$ (resp. $\rm (ii)\Rightarrow \rm  (i)$)
 for some constant $C''$ (resp. $C$) depending only on 
  $C$ (resp. $C''$).
 \end{pro}
    \begin{proof} Using the equivalence with \eqref{sub0},  one checks easily that (i)   is essentially equivalent to:
      \begin{itemize}
   \item[{(i)'}] The system $(\varphi^{[k]}_n)$ ($1\le n\le N, k\in \N$)
   satisfies \eqref{1} for some possibly different constant $C'$ depending only on $C$.    \end{itemize}
   Then $\rm (i)'\Rightarrow \rm (ii)$ by Talagrand's \eqref{3}
   and Proposition \ref{p2}.
   The converse  $\rm (ii)\Rightarrow \rm  (i)$ 
   follows from \eqref{sub2} and  
   Proposition \ref{p2}.
   Indeed, \eqref{g2} with $p=1$ applied to $(\varphi^{[k]}_n)$
   (with a suitable choice of $x_n \in \ell_\infty$)
   yields for $Z=\sum a_n \varphi_n$ and $S=\sum a_n g_n$
   $$\E \sup\nolimits_k   |Z^{[k]}|(\log k)^{-1/2}     \le C \E  \sup\nolimits_k   |S^{[k]}|(\log k)^{-1/2} .$$
   By \eqref{sub2} this implies (i).
  \end{proof}

\section{Systems of random matrices}\label{s3}

Assume given a sequence of finite dimensions ${d_n}$. 

From now on $g_n$ will be an independent sequence
of random $d_n\times d_n$-matrices, such that $\{ { {d_n}}^{1/2} g_n(i,j) \mid 1\le i,j\le {d_n}\}$ are i.i.d. normalized $\CC$-valued Gaussian random  variables. Note $\|g_n(i,j)\|_2={d}_n^{-1/2}$.

For each $n$
let $(\varphi _n)$ be a  random matrix of size ${d_n}\times {d_n}$
on $(T,m)$. We call this a ``matricial system".
We will compare $(\varphi _n)$  with the sequence
$(u_n)$  that is an independent sequence
where each $u_n$ is uniformly distributed over the unitary group
$U({d_n})$.
The subGaussian condition  becomes: for any $N$ and
  $y_n\in M_{d_n}$ ($n\le N$) we have

 \begin{equation}\label{11}\|\sum {d_n} \tr (y_n \varphi _n)\|_{\psi_2}\le C (\sum   {d_n}\tr |y_n|^2)^{1/2} =\|\sum {d_n} \tr (y_n g_n)\|_{2}. \end{equation}
 In other words, $\{d_n^{1/2}\varphi_n(i,j)\mid n\ge 1, 1\le i,j\le d_n\}$
 is a $C$-subGaussian system of functions.
The uniform boundedness assumption becomes
\begin{equation}\label{31}
 \exists C' \ \forall n\quad \|\varphi_n\|_{L_\infty(M_{d_n})} \le C'.
\end{equation}
As for the orthonormality condition
it becomes 
\begin{equation}\label{32}\int \varphi_n(i,j) \ovl{\varphi_{n'}(k,\ell)}=d_n^{-1} \delta_{n,n'}\delta_{i,k}\delta_{j,\ell}.\end{equation}
In other words, $\{d_n^{1/2}\varphi_n(i,j)\mid n\ge 1, 1\le i,j\le d_n\}$
 is an orthonormal system.

This is modeled on the case when $(\varphi_n)$ is a sequence of
distinct irreducible representations on a compact group.

Actually, we will  consider a slightly more general situation.
We assume that there are complex-valued  $\{\psi_n(i,j)\mid n\ge 1, 1\le i,j\le d_n\}$
in $L_\infty(m)$ such that
\begin{equation}\label{33}\sup\nolimits_n\| \psi_n\|_{L_\infty(m;M_{d_n} )}\le C'\end{equation} and 
\begin{equation}\label{37}\int \varphi_n(i,j)\ovl{ \psi_{n'}(k,\ell)}\ dm=d_n^{-1} \delta_{n,n'}\delta_{i,k}\delta_{j,\ell}.\end{equation}
Equivalently  
\begin{equation}\label{37b}
\int \varphi_n \otimes \ovl{\psi_{n'}}\ dm =0 \text{ if } n\not=n' \text{ and }
\int \varphi_n \otimes \ovl{\psi_{n} }\ dm =  d_n^{-1}\sum\n_{i,j\le d_n} e_{ij} \otimes e_{ij}.\end{equation}
Applying transposition on the second factor this is also equivalent to
\begin{equation}\label{37c}
\int \varphi_n \otimes {\psi^*_{n'}} \ dm=0 \text{ if } n\not=n' \text{ and }
\int \varphi_n \otimes  {\psi^*_{n} }\ dm=  d_n^{-1}\sum\n_{i,j\le d_n} e_{ij} \otimes e_{ji}.\end{equation}
Note that
\begin{equation}\label{37d}
\int \varphi_n \otimes  {\psi^*_{n} }\ dm=  d_n^{-1}\sum\n_{i,j\le d_n} e_{ij} \otimes e_{ji}\Leftrightarrow \forall 
a
\in M_{d_n}\quad 
\int \varphi_n a  {\psi^*_{n} }\ dm =  d_n^{-1}\tr(a) I ,\end{equation}
and    also
\begin{equation}\label{37e}
 \Leftrightarrow\forall 
a
\in M_{d_n}\quad 
\int \psi_n a  { \varphi^*_{n} }\ dm =  d_n^{-1}\tr(a) I .\end{equation}   
Thus, if both \eqref{31} and  the orthonormality  \eqref{32} 
 hold, then \eqref{33} and \eqref{37}  hold
for the choice $\psi_n=   {\varphi_n}$.
In any case, we will conclude from this
(see Corollary \ref{69'}) that   $\exists\alpha$ such that
for any   $(a_n)$ with $a_n\in M_{d_n}$ 
\begin{equation}\label{12}\sum {d_n} \tr |a_n|  \le \alpha \sup_{(t_1,t_2)\in T\times T}
|\sum {d_n} \tr (a_n \psi _n(t_1)\psi _n(t_2) )|
   .\end{equation}

  Let $U(d)$ denote the (compact) group of unitary $d\times d$ matrices.\\
  In the next Lemma,  we
   give   a   simple argument from \cite{MaPi} showing that 
   the family $\{u_k(i,j)\} $ ($k\ge 1$, $1\le i,j\le d_k$) is dominated by $\{g_k(i,j)\}$ ($k\ge 1$, $1\le i,j\le d_k$),  using an explicit positive operator $T$, bounded
   on $L_p$ for all $1\le p\le\infty$. By work due to
   Fig\`a-Talamanca and   Rider,
     this family has long been known to be subGaussian,   see \cite[\S 36, p. 390]{HR}. The idea in Lemma \ref{R72} was used in    \cite{MaPi} to give  a simpler proof of the latter fact.
   
  \begin{lem}\label{R72}  Let $(d_k)_{k\in I}$ be an arbitrary collection of integers.
   Let $G=\prod_{k\in I} U(d_k)$. Let $u\mapsto u_k$
   denote the coordinates on $G$, and $u_k(i,j)$  ($1\le i,j\le d_k$)
   the entries of $u_k$.
   Let $\{g_k(i,j)\}$ ($1\le i,j\le d_k$) be a collection of
   independent complex valued  Gaussian  random variables
  such that $\E(g_k(i,j))=0$ and $\E|g_k(i,j)|^2=1/d_k$, on a probability space
  $(\Omega,\P)$. For some $C_0>0$ there is  an operator
  $T: \ L_1(\Omega,\P) \to L_1(G,m_G)$ with
  $ \|T: \ L_p(\Omega,\P) \to L_p(G,m_G)\|\le C_0 $ for all $1\le p\le\infty$
  such that
  $$\forall k \forall  i,j\le d_k\quad T(g_k(i,j))= u_k(i,j).$$
\end{lem}
  \begin{proof} 
  Let $g_k=v_k |g_k|$ be the polar decomposition of $g_k$.
  The key observation
is that $(v_k)$ and $(|g_k|)$ are independent random variables,
and that $(v_k)$ and $(u_k)$ have the same distribution.
Also for any fixed $v\in U(d_k)$, $ g_k$ has the same distribution as
  $(v g_kv ^{-1})$, and hence $( |g_k|)$ has the same distribution as
  $(v |g_k|v^{-1})$.
   Let $V$ denote the conditional expectation with respect to $(v_k)$ on $(\Omega, \P)$.
  Then $V( g_k)= v_k \E|g_k|$. Since $\E|g_k| $ commutes
  with any $v\in U(d_k)$, we have $\E|g_k| =\d_k I$
  for some $\d_k>0$. We claim that $\d=\inf\n_k \d_k>0$.
  Indeed, let  $c_1=d_k^{1/2}\E|g_k(i,j)|$. Note that $c_1$
  is independent of $(k,i,j)$.
  We have $(\sum\n_{i,j}(\E|g_k(i,j)|)^2)^{1/2} \le \E (\sum\n_{i,j}|g_k(i,j)|^2)^{1/2}=\E(\tr|g_k|^2)^{1/2}$ and hence
  $$ d_k^{1/2} c_1  \le \E(\tr|g_k|^2)^{1/2} \le \E (\tr |g_k|  \|g_k\|)^{1/2}
  \le  ( \E \tr |g_k| )^{1/2}(\E   \|g_k\|)^{1/2} = (\d_k d_k)^{1/2} (\E   \|g_k\|)^{1/2} $$
  and since, as is well known $\sup\n_k \E   \|g_k\|<\infty$,
  the claim $\d >0$ follows.\\
  Since $(v_k)$ and $(u_k)$ have the same distribution
  we can identify  $V$ to an operator $V_1:\  L_1(\Omega,\P)
 \to L_1(G,m_G)$ such that  $V_1( g_k(i,j))=   \d_k u_k(i,j)$.
  We will now modify $V_1$ to replace $\d_k$ by $\d$. Let 
  ${\cl E}_k:\ L_1(G,m_G)\to L_1(G,m_G)$ denote the conditional expectation with respect to the $\sigma$-algebra generated by the   coordinates
  $\{u_j\mid j\not= k\}$
   on $G$, so that ${\cl E}_k(u_j)=u_j$ if $ j\not= k$, and ${\cl E}_k(u_k)=0$.
    Let $Id$ denote the identity on $L_1(G,m_G)$.
  Recall $0<\d/\d_k\le 1$.  Then let $$W=\prod ( (1-\d/\d_k){\cl E}_k+ (\d/\d_k) Id ).$$ By a simple limiting argument, this infinite product makes sense and defines
an operator $W:\ L_p(G,m_G)\to L_p(G,m_G)$ with $\|W\|\le 1$
for any $1\le p\le\infty$,
such that $W(u_k(i,j))=(\d/\d_k) u_k(i,j)$.
Thus, setting $T= \d^{-1} WV_1$ we have
$$T(g_k(i,j))=\d^{-1} W(V_1(g_k(i,j)) ) =  u_k(i,j),$$
and $\|T:\ L_p(G,m_G)\to L_p(G,m_G)\|\le C_0=\d^{-1}$.
In addition, note that $T$ is actually positive.
  \end{proof}
  The following basic fact  compares the notions
 of randomly Sidon for $(g_k)$ and $(u_k)$. It is proved by the same truncation trick
 that was used in \cite{Pi}. See \cite[Chap.V and VI]{MaPi} for further details and more general facts.   
 \def\z{u}
 \begin{lem}\label{68} Let $\psi_n\in L_\infty(m; M_{d_n})$ ($n\ge 1$)
be an arbitrary matricial system satisfying \eqref{33}.
The following are equivalent:
\begin{itemize}
\item[(i)] There is a constant $\alpha_1$ such that 
for any $n$ and any $x_k\in M_{d_k}$ 
$$\sum d_k \tr|x_k| \le \alpha_1 \E\| \sum d_k \tr(x_k g_k  \psi_k)\|_\infty.$$
\item[(ii)] There is a constant $\alpha_2$ such that 
for any $n$ and any $x_k\in M_{d_k}$ 
$$\sum d_k \tr|x_k| \le \alpha_2 \int \| \sum d_k \tr( x_k \z_k  \psi_k)\|_\infty m_G(d\z).$$
where $\z=(\z_n)$ denotes (as before) a random sequence of
unitaries uniformly distributed  in $  G= \prod\n_{n\ge 1} U(d_n)$.
\end{itemize}
  \end{lem}
  \begin{proof}[Sketch]  
  From Lemma \ref{R72} it is easy to deduce that
  $$  \int \| \sum d_k \tr( x_k  \z_k \psi_k)\|_\infty m_G(d\z)
  \le C_0\E\| \sum d_k \tr(x_k g_k  \psi_k)\|_\infty,$$
  and hence (ii) $ \Rightarrow$ (i).
 To check the converse, recall the well known fact that 
 $c_4=\sup \E \|g_n\|^2 <\infty,$
 from which it is easy to deduce by Chebyshev's inequality that there
 exists   $c_5>0$ such that
 $$\sup \E(\|g_n\| 1_{\{\|g_n\|> c_5\}}  \le (2\alpha_1C')^{-1}.$$
  We may assume that the sequences $(\z_n)$
  and $(g_n)$ are mutually independent.
  Then the sequences $(g_n)$ and $(\z_n g_n)$
   have the same distribution. 
Then by the triangle inequality  and by Remark \ref{90}
  $$\E\| \sum d_k \tr(x_k g_k  \psi_k)\|_\infty
  =\E\| \sum d_k \tr(x_k u_k g_k  \psi_k)\|_\infty$$
  $$
  \le \E\| \sum d_k \tr(x_k u_k g_k 1_{\{\|g_k\|\le c_5\}} \psi_k)\|_\infty
  +\E\| \sum d_k \tr(x_k u_k g_k 1_{\{\|g_k\|> c_5\}} \psi_k)\|_\infty$$
  $$ \le c_5 \E\| \sum d_k \tr(x_k \z_k   \psi_k)\|_\infty
  + (2\alpha_1C')^{-1}  \sum d_k \tr|x_k| \| \psi_k  \|_\infty \qquad\qquad\qquad\ \ $$
  $$ \le c_5 \E\| \sum d_k \tr(x_k \z_k   \psi_k)\|_\infty
  + (2\alpha_1 )^{-1}  \sum d_k \tr|x_k|  . \qquad\qquad\qquad\qquad\qquad\ 
  $$
  Using this we see that 
  (i)  implies
   $$\sum d_k \tr|x_k| \le \alpha_1c_5 \E\| \sum d_k \tr(x_k \z_k   \psi_k)\|_\infty
  +(1/2)  \sum d_k \tr|x_k|  , $$
  and hence (i) $ \Rightarrow$ (ii) with $\alpha_2\le 2\alpha_1c_5$.
    \end{proof}
\begin{dfn}\label{d3'} Let $(\varphi_n)$ be a sequence  with $\varphi_n\in L_\infty(T,m; M_{d_n})$ for all $n$
and let $C>0$.
\begin{itemize}
\item[(i)]  We say that    $(\varphi_n)$
  is   Sidon  with constant $C$
if for any $n$ and any   sequence $(x_k)$ with $x_k\in M_{d_k}$ we have
$$\sum\n_1^n d_k \tr |x_k|\le C \|\sum\n_1^n d_k \tr (x_k \varphi_k)\|_\infty.$$

\item[(ii)]   We say that $(\varphi_n)$ is randomly Sidon   with constant $C$
if for any $n$ and any $x_k\in M_{d_k}$ we have
$$\sum\n_1^n d_k \tr |x_k|\le C \E\|\sum\n_1^n d_k \tr ( x_k g_k \varphi_k)\|_\infty.$$
By Lemma \ref{68} this is equivalent to the previous definition
with random unitaries  $(u_k)$ in place of $(g_k)$.\\
If this holds only for scalar matrices (i.e. for $x_k\in \CC I_{d_k}$)
we say that $(\varphi_n)$ is randomly central  Sidon   with constant $C$.
\item[(iii)]   Let $k\ge 1$.
We say that $(\varphi_n)$ is ${\dot\otimes}^k$-Sidon   with constant $C$
if the system $\{\varphi_n(t_1)\cdots \varphi_n(t_k)\}$ 
  is Sidon 
with constant $C$.\\
We say that $(\varphi_n)$ is randomly ${\dot\otimes}^k$-Sidon   with constant $C$
if $\{\varphi_n(t_1)\cdots \varphi_n(t_k)\}$ 
  is randomly Sidon 
with constant $C$.
\end{itemize}
Now assume merely that $\{\varphi_n\}\subset L_2(T,m)$.
\begin{itemize}
\item[(iv)] 
We say that $(\varphi_n)$ is subGaussian with constant $C$
(or $C$-subGaussian) if   
for any $n$ and any complex sequence $(x_k)$ 
we have
$$\| \sum\n_1^n  d_k\tr( x_k \varphi_k)\|_{{\psi_2}(m)} \le C(\sum\n_1^n d_k\tr|x_k|^2)^{1/2}.$$

\end{itemize}
\end{dfn}
  \begin{rem}\label{bat2} Using $(u_k)$ for the randomization it is clear
  that Sidon implies randomly Sidon (with at most the same constant).
  A fortiori, $\dot\otimes^k$-Sidon implies randomly $\dot\otimes^k$-Sidon.
  \end{rem}
  \begin{rem} We should emphasize that central  Sidon
  does not imply randomly central  Sidon, in contrast with the preceding remark.
  \end{rem}
As earlier, we will consider
the following more general form
of the assumption \eqref{11}:
 \begin{equation}\label{11nb}\text{There are } C>0 
 \text{ and   }  u:\ L_1(\P)\to L_1(m)    
  \text{ such  that } \|u\|\le C \text{ and }
 \forall n,i,j \ \ u(g_n(i,j) )=\varphi_n(i,j) . \end{equation}
 In other words, the $\varphi_n$'s are  entrywise
 $C$-dominated by the $g_n$'s.
 
 Here again, Talagrand's inequality \eqref{3} is crucial. 
 Restated 
in the present context:
 \begin{thm}\label{87} For any matricial system
 $(\varphi_n)$, 
 \eqref{11} $\Rightarrow$  \eqref{11nb} (possibly with a different   $C$).
 \end{thm}

 \noindent{\bf Notation}. Let $(T_1,m_1)$
 and $(T_2,m_2)$ be probability spaces.
 Let $\psi^1 \in L_\infty(T_1,m_1; M_d)$,
 $\psi^2 \in L_\infty(T_2,m_2; M_d)$.
 We denote by
 $$\psi^1 \dot\otimes \psi^2\in L_\infty(T_1\times T_2,m_1 \times m_2; M_d)$$
 the function defined on $T_1\times T_2$ by
 $$\psi^1 \dot\otimes \psi^2(t_1,t_2)= \psi^1 (t_1) \psi^2 (t_2).$$
 We now state the matricial generalization of Corollary \ref{cbl}.
     
\begin{thm}\label{t2} Assuming \eqref{33} and \eqref{37},
 we have $\eqref{11} \Rightarrow \eqref{12} .$\\
More generally,  given two systems $(\varphi^1_n),(\varphi^2_n)$
 satisfying \eqref{11nb} with respective constants $C_1,C_2$, and two systems $(\psi^1_n),(\psi^2_n)$
 satisfying \eqref{33}  with respective constants $C'_1,C'_2$ and  such that the pairs
 $(\varphi^1_n),(\psi^1_n)$ and  $(\varphi^2_n),(\psi^2_n)$
 satisfy 
 \eqref{37}, 
 the system $( \psi^1 _n \dot\otimes\psi^2 _n  )$ is Sidon with
 a constant depending only on 
  $C_1,C_2,C'_1,C'_2 $.  
  \end{thm}
 Let $v\in L_1\otimes L_1$.
We denote 
$$\gamma_2^*(v)=\sup\{ |\sum\nolimits_1^N \langle v, x_j\otimes y_j \rangle|
\mid \sum\nolimits_1^N x_j \otimes y_j\in L_\infty\otimes L_\infty,
\| (\sum\nolimits_1^N |x_j|^2)^{1/2}  \|_\infty \|  (\sum\nolimits_1^N |y_j|^2)^{1/2} \|_\infty \le 1 \ \}.$$
Theorem \ref{t2} will be deduced rather easily from Theorem \ref{nt1}
using the following simple fact.

\begin{lem}\label{nl1}
Let $v\in L_1(m_1)\otimes L_1(m_2)$ be a tensor
such that $\gamma_2^*(v)\le 1$. Let
$x\in M_d$ and let $\psi^j\in L_\infty(m_j; M_d)$.
Let $f(t_1,t_2)=\tr(a\psi^1\psi^2)\in L_\infty(m_1\times m_2)$.
Then
$$|\langle v, f\rangle|\le  \tr|a| \|  \psi^1  \|_{L_\infty(m_1; M_d)}
\|  \psi^2  \|_{L_\infty(m_2; M_d)}.$$
\end{lem}
\begin{proof} 
We may assume (by polar decomposition)  $a=a_2a_1$ with 
$\tr|a_1|^2=\tr|a_2|^2= \tr|a|$.
  Then 
$f={\tr([a_1 \psi^1] [\psi^2 a_2])}= \sum\nolimits_{k,\ell} (a_1\psi^1)(\ell,k) \otimes (\psi^2 a_2)(k,\ell)$.
Then
$$\langle v, f\rangle  =
\sum\nolimits_{k,\ell}      \langle v, (a_1\psi^1)(\ell,k) \otimes       (\psi^2 a_2)(k,\ell)   \rangle $$
and hence since $\gamma_2^*(v)\le 1$
$$|\langle v, f\rangle |\le
\|(\sum\nolimits_{k,\ell }|    (a_1\psi^1)(\ell,k)  |^2 )^{1/2}\|_\infty
\|(\sum\nolimits_{k,\ell }|  (\psi^2 a_2)(k,\ell)   |^2 )^{1/2} \|_\infty
$$
 but we have $$(\sum\nolimits_{k,\ell }|    (a_1\psi^1)(\ell,k)  |^2 )^{1/2}=(\tr|a_1\psi^1|^2)^{1/2} \le (\tr|a_1|^2)^{1/2}\|\psi^1\|_{M_d} $$
 and similarly for $\psi^2 a_2$. Thus we obtain
 $$|\langle v, f\rangle |\le (\tr|a_1|^2)^{1/2}(\tr|a_2|^2)^{1/2} \|\psi^1\|_{L_\infty(M_d)}  \|\psi^2\|_{L_\infty(M_d)} =(\tr|a|) \|\psi^1\|_{L_\infty(M_d)}  \|\psi^2\|_{L_\infty(M_d)}$$ proving the Lemma.
\end{proof}
\begin{rem}\label{nr1} By Grothendieck's
well known inequality (see e.g. \cite[Th. 2.1]{Pig})
we have $\gamma_2^*(v) \le K_G \|v\|_\vee$ for any $v\in L_1(m_1)\otimes L_1(m_2)$. Thus in Theorem \ref{nt1} we have
$\gamma_2^*(r) \le K_G \|r\|_\vee\le  K_G \d$.
But actually a close examination (see Remark \ref{melo}) shows
that we directly obtain a bound for $\gamma_2^*(r)$
without recourse to Grothendieck's theorem. Indeed, with the notation of the proof of Theorem \ref{nt1}, one has 
$\gamma_2^*(R)  \le  \|T_\d(I-P_1):\ L_2\to L_2\|\|u_1\|\|u_2\|.$
\end{rem}
\begin{rem} In case the reader is wondering about that, the general definition
of the $\gamma_2^*$-norm for an element $r $ in  the algebraic tensor  product $X\otimes Y$ of two Banach spaces
is $$\gamma_2^*=\inf\{ \|a\|_{M_n}(\sum\nolimits_1^n \|x_j\|^2)^{1/2} (\sum\nolimits_1^N \|y_j\|^2)^{1/2} \},$$
where the infimum runs over all possible ways to write $r$
as $r=\sum\nolimits_{1\le i,j\le n} a_{ij} x_i\otimes y_j $ 
($n\ge 1$, $x_i\in X$, $y_j\in Y$, $a_{ij}\in \CC$). When $X=Y=L_1$
this is identical to the preceding definition (see \cite{Pig}).
\end{rem}
\begin{proof}[Proof of   Theorem \ref{t2}] 
We will apply
Theorem \ref{nt1}.\\
By homogeneity, we may assume that $(\varphi^1_n)$
and $(\varphi^2_n)$ satisfy \eqref{11nb} with $C_1=C_2=1$
(then $C'_j$ is replaced by $C_jC'_j$ and $\psi^j_n$ by $C_j\psi^j_n$).
Let $V_n$ be arbitrary in the unit ball of $M_{d_n}$.
Consider the tensor
$$S=\sum\nolimits_n d_n \sum\nolimits_{i,k,\ell}   V_n(i,k) \varphi^1_n(k,\ell) \otimes \varphi^2_n(\ell,i),$$
which roughly could be written as $\sum\nolimits_n d_n \tr(V_n\varphi^1_n\varphi^2_n)$ using tensor product to form the products of 
matrix coefficients in $\varphi^1_n$ and $ \varphi^2_n$.
Let $\varphi'^1_n=V_n \varphi^1_n$ (this denotes product of the scalar matrix
$V_n$ by the $L_1$-valued matrix $\varphi^1_n$).
Note that,
by \eqref{c2} 
(and Proposition \ref{p2}) applied to the standard normal family $\{d_n^{1/2}\varphi^1_n(i,j)\mid 1\le n\le N, i,j\le d_n\}$,
 if we replace $(\varphi^1_n)$ by $(\varphi'^1_n)$, then 
\eqref{11nb}    
still holds.
 This gives us
$$S=\sum\nolimits_n  \sum\nolimits_{i, \ell}   [d^{1/2}_n \varphi'^1_n(i,\ell) ]\otimes [d^{1/2}_n\varphi^2_n(\ell,i)].$$
By   
Theorem \ref{nt1} (actually we could invoke Corollary \ref{cor1}), and
using   Remark \ref{nr1},
this shows that we have a decomposition
$S=t+r$ with $\|t\|_\wedge\le w(\d)$ and
$\gamma_2^*(r) \le \d$.
Now
let $f=\sum d_n \tr(a_n \psi^1_n  \psi^2_n)\in L_\infty(m_1)\otimes L_\infty(m_2)$, or more explicitly
$$f=\sum\nolimits_n d_n \sum\nolimits_{i,k,\ell}  a_n(i,k)  \psi^1_n(k,\ell)  \otimes \psi^2_n(\ell,i).$$
Recalling \eqref{37} (denoting simply $\|f\|_\infty=
\|f\|_{L_\infty(m_1 \times m_2)}$) we find 
$\langle S, f\rangle =\sum d_n   \tr( {}^t V_n a_n)$.
Then $$ |\langle t, f\rangle|\le \|t\|_\wedge \|f\|_\infty \le w(\d) \|f\|_\infty$$
and by Lemma \ref{nl1} and \eqref{33}
$$ |\langle r, f\rangle|\le \d \sum d_n \tr|a_n| \|\psi^1_n\|_{L_\infty(m_1; M_{d_n})} \|\psi^2_n\|_{L_\infty(m_2; M_{d_n})}
\le \d \sum d_n \tr|a_n| C'_1C'_2.$$
Therefore
$$|\sum d_n   \tr( {}^t V_n a_n)|=|\langle S, f\rangle| \le w(\d) \|f\|_\infty+ \d \sum d_n \tr|a_n| C'_1C'_2,$$
and taking the sup over all $V_n$'s we find
$\sum d_n \tr|a_n| \le w(\d) \|f\|_\infty+ \d \sum d_n \tr|a_n| C'_1C'_2,$
and we conclude choosing $\d$ small enough so that $ \d  C'_1C'_2<1$
that we have
$$\sum d_n \tr|a_n| \le w(\d) (1-   \d  C'_1C'_2)^{-1} \|f\|_\infty.$$
Taking $\d  C'_1C'_2=1/2$, this completes the proof with 
$\alpha=2w( (2C'_1C'_2)^{-1})$, and using Remark \ref{mel}
we obtain the announced bound on $\alpha$.
\end{proof}

In particular, we have
 \begin{cor}\label{69'} 
Let $\{\psi_n\mid 1\le n\le N\} \subset L_\infty(m;M_{d_n})$
satisfying \eqref{33}.
Assume that the system
$\{d_n^{1/2} \psi_n(i,j)\mid 1\le n\le N, 1\le i,j\le d_n\}$
admits a biorthogonal system that is $1$-dominated by
$\{d_n^{1/2} g_n(i,j)\mid 1\le n\le N, 1\le i,j\le d_n\}$.
Then 
there is a number $\alpha=\alpha(C')$
(depending only on  $C'$)  
such that
$(\psi_n)$ is ${\dot\otimes}^2$-Sidon with 
constant $\alpha$.  
\end{cor}
\begin{rem}[Returning to group representations]\label{86}
 Let $G$ be a compact group.
  Let $\Lambda=\{\pi_n\}\subset \hat G$ be a sequence of distinct unitary representations
  on  $G$. Let $d_n=\dim(\pi_n)$. Then (Peter-Weyl) $\{d_n^{1/2}\pi_n(i,j)\}$
  is an orthonormal system in $L_2(G)$.
  Thus we may apply Corollary \ref{69'}  with
  $\psi_n=\varphi_n=\pi_n$ on $(G,m_G)$. Recalling Theorem \ref{87},
  we find that if $\Lambda=\{\pi_n\}$ satisfies \eqref{11},
  then $\Lambda$ is a Sidon set.
  Indeed, for  representations, ${\dot\otimes}^2$-Sidon (or ${\dot\otimes}^k$-Sidon)  obviously implies Sidon. This was first proved in \cite{Pi,Pi2}.
\end{rem}
   \begin{rem}[On almost biorthogonal systems]\label{71'}
  In the situation of the preceding Corollary, just like in Remark
  \ref{R71} it suffices
  to have a system almost biorthogonal to 
  $\{d_n^{1/2} \psi_n(i,j)\mid 1\le n\le N, 1\le i,j\le d_n\}$.
  More precisely, let $I=\{(n, i,j) \mid   1\le n\le N, 1\le i,j\le d_n\}$.
  Let $p=(n, i,j)\in I$ and $p'=(n', i',j')\in I$.
  Let $a=[a(p,p')]$ be the matrix defined by
  $a(p,p') =\langle  d_n^{1/2} \varphi_n(i,j) , d_{n'}^{1/2} \psi_{n'}(i',j') \rangle$.
  Assume $a$ invertible with inverse $b$ such that $\|b\|\le c$.
  If $\{d_n^{1/2} \varphi_n(i,j)\mid 1\le n\le N, 1\le i,j\le d_n\}$
  is $1$-dominated by
$\{d_n^{1/2} g_n(i,j)\mid 1\le n\le N, 1\le i,j\le d_n\}$,
 then 
there is a number $\alpha=\alpha(c,C')$
(depending only on  $c,C'$)  
such that
$(\psi_n)$ is ${\dot\otimes}^2$-Sidon with 
constant $\alpha$.  
\end{rem}
\begin{rem}[On almost biorthogonal single systems]\label{71''}
The preceding Remark is already significant for a single random
$d\times d$-matrix $\psi\in L_\infty(m;M_d)$ with
$\|\psi\|_{L_\infty(m;M_d)}\le C'$. Indeed, let $\varphi\in L_1(m;M_d)$.
Let $\{d^{1/2}   g_{ij}\mid 1\le i,j \le d\}$ be a family 
of Gaussian complex variables with   $\E(g_{ij})=0$ and  
$\E|g_{ij}|^2=1$.
We again replace \eqref{37} by  almost orthogonality.
Let $a$ be the $d^2\times d^2$-matrix defined by
$a(i,j ; i',j')= d \int \psi^*_{ij}  \varphi_{i'j'} \ dm$.
If $a$ is invertible and $\|a^{-1}\|_{M_{d^2} }\le c$,
and if $(d^{1/2}   \varphi_{ij})$ is $1$-dominated by 
$(d^{1/2}   g_{ij})$, there is $\alpha(c,C')$ \emph{(independent of $d$)}
such that
$$\forall x\in M_d\quad \tr |x|\le \alpha \|tr (x(\psi \dot\otimes \psi))\|_\infty.$$
In other words, the singleton $\{\psi\}$ is ${\dot\otimes}^2$-Sidon
with constant $\alpha$.
\end{rem}

\section{An example}\label{exam}

The following example  provides us with an illustration
of the possible use of Corollary \ref{69'} and Remark \ref{71''}.
Although there may well be an   alternate argument, we do not see
a direct proof of the phenomenon appearing in Corollary \ref{exa}.

Let $\chi\ge 1$ be a constant (to be specified later).
Let $T_n$ be the set of $n \times n$-matrices $a=[a_{ij}]$
with $a_{ij}=\pm 1/\sqrt n$.
Let
 $$A^\chi_n=\{a\in T_n \mid \|a\| \le \chi\}.$$
This set includes the famous Hadamard matrices.
We have then
\begin{cor}\label{exa} There is a numerical $\chi\ge 1$  such that 
for some $C$ we have
$$\forall n\ge 1\  \forall x\in M_n\quad \tr|x| \le C \sup_{a',a''\in A^\chi_n}  |\tr(x a' a'')|.$$
Equivalently, denoting
the set  $\{a' a''\mid a',a''\in A^\chi_n\}$ by $A^\chi_nA^\chi_n$, its absolutely convex hull satisfies$$  (\chi)^2 \text{absconv}[A^\chi_nA^\chi_n]  \subset B_{M_n} \subset  C \text{absconv}[A^\chi_nA^\chi_n]$$
\end{cor}
\begin{proof} Let $(\Omega,\P)$ be a probability space.
Let $g^\R_n$ (resp. $g^\C_n$) be a random $n \times n$-matrix 
on $(\Omega,\P)$ with i.i.d. 
 real-valued (resp. complex-valued) normalized Gaussian entries of mean $0$ and $L_2$-norm $=(1/n)^{1/2}$
as usual.  Let $$\vp_n(i,j)(\omega)=n^{-1/2}{\rm sign}(g^\R_n(i,j)(\omega)).$$
Let  $\gamma(1)$
be the $L_1$-norm of a normal Gaussian variable 
(i.e. $\gamma(1)=(2\pi)^{-1/2}\int |x| \exp{-(x^2/2)} dx $).
Clearly, for any $i,j, i',j'$
$$\E(\vp_n(ij) g^\R_n(i'j')) =\d_{ii'} \d_{jj'} n^{-1/2}\E|g^\R_n(i'j')|=\d_{ii'} \d_{jj'} n^{-1} \gamma(1) .$$

We define $\psi_n\in L_\infty(\P; M_n)$ by
$$\psi_n = (2^{1/2}/ \gamma(1) ) \vp_n  1_{\vp_n\in A^\chi_n} .$$
Note $\|\psi_n\|_{M_n}\le 2^{1/2}\chi/\gamma(1)$. 
We will use Corollary \ref{69'}.
Clearly $ \E( \psi^*_n \otimes g_{n'})=0$ whenever $n\not=n'$. To handle the case
$n=n'$,
it is well known that there is  $c_0>0$ such that
$$\forall n\forall \chi \ge c_0\quad
\P(\vp_n\not\in  A^\chi_n)= \P\{\| \vp_n\|>\chi\} \le \exp - (c_0 n \chi^2).$$
Therefore, if $\chi \ge c_0$ for any $i,j, i',j'$  we have
$$|\E(\psi_n(ij) 2^{-1/2} g^\R_n(i'j'))-\d_{ii'} \d_{jj'} n^{-1}   | \le
\gamma(1)^{-1} \E(|g^\R_n(i'j')| 1_{\{\vp_n\not\in  A^\chi_n\}})\le \gamma(1)^{-1} n^{-1/2}  \exp - (c_0 n \chi^2/2).$$
Fix $\chi \ge c_0$.
This shows that the matrix
$a(i,j ; i',j')= n  \int \psi_n(ij) 2^{-1/2} g^\R_n(i'j') \ dm$ is a perturbation
of the identity when  $n$ is large enough
so that (say) when $n\ge n_0(\chi)$ it is invertible with inverse
of norm $\le 2$. Let $\varphi_n(ij)= ( 2^{-1/2} g^\R_n(ij))$.
 Note that $(\varphi_n(ij))$ is obviously $1$-dominated by
 $(   g^\C_n(ij))$.
 The conclusion follows from Remark \ref{71''} (applied here with $d=n$
 and $C'=2^{1/2}\chi/\gamma(1)$) for all $n\ge n_0(\chi)$. But
 the case $n<n_0(\chi)$ can be handled trivially
 by adjusting $\alpha$.
\end{proof}

\section{Randomly Sidon matricial systems}\label{rsms}

We first recall a useful basic fact (see \cite{MaPi} for variations on this theme).
  \begin{rem}[Contraction principle]\label{90} Let $(u_k)$ and $G$ be as in Lemma \ref{R72}.
  Let $\{x_k(i,j)\mid k\ge 1, 1\le i,j\le d_k\}$ be a finitely supported family 
  in an arbitrary Banach space $B$. For any matrix $a\in M_{d_k}$
  with complex entries, we denote by $ax$ and $xa$
  the matrix products (with entries in $B$).
  By convention, we write $\tr(u_k x_k)=\sum\n_{ij} u_k(i,j)x_k(j,i)$.
  With this notation, the following``contraction principle" holds
  $$\int \|\sum d_k \tr( a_k u_k b_k x_k)\| dm_G
 \le \sup\n_k\|a_k\|_{M_{d_k}} \sup\n_k \|b_k\|_{M_{d_k}} 
 \int \|\sum d_k \tr(   u_k    x_k)\| dm_G.
 $$
 Indeed, this is obvious by the translation invariance of $m_G$
 if $a_k,b_k$ are all unitary. Then the result follows by an extreme point argument, since the unit ball of ${M_{d_k}} $ is the closed convex hull of
 its unitary elements.\\
 The same inequality (same proof) holds with $(g_k)$ (as in Lemma \ref{68}) in place of $(u_k)$.
   \end{rem}
 
 \begin{pro}\label{46} Let $( \psi^1_n)$ be a randomly central Sidon system
 with constant $C_1$.
 Let $(\varphi^2_n,\psi^2_n)$ be a system satisfying \eqref{33} with constant
 $C'_2$
 and \eqref{37d}.
 Then the system  
 $ (\psi^1_n  \dot\otimes \varphi^2_n)$ is randomly Sidon with constant
 $C_1C'_2$.
 \end{pro}
 \begin{proof} Let $x_k\in M_{d_k}$.
 We have (for simplicity in the sequel we always abusively write $\sup$ for essential suprema)
 $$\E \sup\n_{t_1,t_2} |\sum d_k \tr ( x_k g_k   \psi^1_k (t_1) \varphi^2_k (t_2)|
 \ge \sup\n_{t_2} \E \sup\n_{t_1} |\sum d_k \tr (x_k g_k \psi^1_k (t_1)  \varphi^2_k (t_2)|.
$$ 
Assume  $(\psi^2_k)$ satisfies  \eqref{33}  with constant $C'_2$.
Then, by Remark \ref{90}, we have for a.a. fixed $t_2$
$$\E \sup\n_{t_1} |\sum d_k \tr (x_k g_k  \psi^1_k (t_1) \varphi^2_k (t_2)|
\ge (C'_2)^{-1}\E \sup\n_{t_1} |\sum d_k \tr (|x_k^*| \psi_k^2(t_2)^* g_k   
\psi^1_k (t_1) \varphi^2_k (t_2)|$$
and by the trace identity, this is
$$=(C'_2)^{-1}\E \sup\n_{t_1} |\sum d_k \tr (g_k   
\psi^1_k (t_1) \varphi^2_k (t_2) |x_k^*| \psi_k^2(t_2)^* |$$
 and hence
 $$\sup\n_{t_2} \E \sup\n_{t_1} |\sum d_k \tr (x_k g_k  
\psi^1_k (t_1) \varphi^2_k (t_2)|\ge
 (C'_2)^{-1} \E \sup\n_{t_1} |\sum d_k \tr ( g_k\psi^1_k (t_1)  [\int   
 \varphi^2_k   |x_k^*|
 {\psi_k^2}^* dm(t_2)]|
$$ 
and by \eqref{37d} and the randomly central Sidon assumption on $(\psi^1_k)$ the last term
is
$$=(C'_2)^{-1} \E \sup\n_{t_1} |\sum   \tr ( g_k  \psi^1_k (t_1)[ \tr |x_k^*|]
|.
\ge (C'_2)^{-1}C_1^{-1}   \sum d_k \tr | x_k^*|.
$$ 
Since $\tr | x_k^*|=\tr | x_k|$ this proves the announced result.
\end{proof}
  \begin{rem} For irreducible representations on a compact group
  Proposition \ref{46} shows that randomly central Sidon
  implies randomly   Sidon with identical constants.
 \end{rem}
\begin{pro}\label{P} Let $( \psi^2_n)$ be a randomly   Sidon system on $(T_2,m_2)$.
 Let $(\varphi^1_n,\psi^1_n)$ be a system satisfying   \eqref{37c}, or equivalently \eqref{37},
 on $(T_1,m_1)$. We also assume that $( \psi^1_n)$, $( \psi^2_n)$ and 
  $(\varphi^1_n)$ are all uniformly bounded, i.e. satisfy \eqref{33}. 
  Then the system  
 $ (\psi^1_n  \dot\otimes \psi^2_n)$ is ${\dot\otimes}^2$-Sidon.
 \end{pro}
 \begin{proof}  
 Let $f_k(t_1,t_2)=\sum\n_{\ell} a_{k\ell} ( \psi^1_\ell (t_1)   \psi^2_\ell(t_2)) b_{k\ell} $
where  $a_{k\ell}$ (resp. $b_{k\ell}$) is a matrix
of size $d_k \times d_\ell$ (resp. $d_\ell \times d_k$).
Assume $( \psi^2_n)$     randomly   Sidon with constant $C_2$,
and $\|\varphi^1_n\|_{L_\infty(M_{d_n})}\le C'_1$ for all $n$.
We claim that
$$\sum d_k |\tr (a_{kk})   \tr( b_{kk})|\le C_2C'_1 
\E\sup\n_{t_1,t_2} |   \sum d_k\tr(g_k   f_k)    |.
$$
By \eqref{37c},  \eqref{37d} and \eqref{37e} we have
$$d_k\int   {\varphi^1_k(t_1)}^* f_k  dm(t_1)  = \tr (a_{kk}) \psi^2_k(t_2)  b_{kk} .$$
Therefore since $( \psi^2_n)$ is randomly   Sidon with constant $C_2$
we have
$$\sum   |\tr (a_{kk})   \tr( b_{kk})|\le C_2 \E\sup\n_{t_2} | \int \sum d_k\tr(g_k  {\varphi^1_k(t_1)}^* f_k  )dm(t_1)   |$$
$$
\le C_2\int
\E\sup\n_{t_2} |   \sum d_k\tr(g_k  {\varphi^1_k(t_1)}^* f_k)    | dm(t_1)
\le C_2\sup\n_{t_1}
\E\sup\n_{t_2} |   \sum d_k\tr(g_k  {\varphi^1_k(t_1)}^* f_k)    |
$$
and by the contraction principle in Remark \ref{90}
$$\le C_2 C'_1\sup\n_{t_1}
\E\sup\n_{t_2} |   \sum d_k\tr(g_k   f_k)    |
$$
and a fortiori
$$\le C_2 C'_1 
\E\sup\n_{t_1,t_2} |   \sum d_k\tr(g_k   f_k)    |.
$$
Thus we obtain  the claim.\\
Let $m=m_1\times m_2$. Let $E\subset L_1(\P; L_\infty(m))$
be the subspace formed of all the functions of the form
$f= \sum d_k\tr(g_k   f_k)$. Let
$\xi:\ E\to \CC$ be the linear form defined by
$$\xi (f)= \sum   \tr (a_{kk})   \tr( b_{kk}).$$
By Hahn-Banach there is an extension  
$\xi' :\ L_1(\P; L_\infty(m)) \to \CC$ with norm $\le C_2 C'_1$.
Since $L_1(\P; L_\infty(m))$ can be identified to the projective tensor product
of $L_1(\P)$ and $L_\infty(m)$, $\xi$ defines
a bounded linear map $T:\ L_1(\P) \to L_\infty(m)^*$ with 
$\|T\| \le C_2 C'_1$ such that
\begin{equation}\label{66} \forall k\forall f_k\quad d_k \langle \sum\n_{i,j}T(g_k(i,j) ) , f_k(j,i)  \rangle=   \tr (a_{kk})   \tr( b_{kk}).\end{equation}
Let $\theta_k(i,j) = T(g_k(i,j) )$
and $\psi_k=  \psi^1_k \dot\otimes\psi^2_k$.
Then \eqref{66} implies $d_k\int \theta_k (j' ,i) \psi_\ell(j, i')=\d_{k\ell}\d_{ij}\d_{i'j'}.$
By Lemma \ref{lr} we may assume that $\theta_k(i,j)\in L_1(m)$
and $\|T\| \le (1+\vp)C_2 C'_1$.
Then we have $\int  \theta_k \otimes \psi_\ell \ dm =0$ if $\ell\not=k$
and 
$$d_k\int  \theta_k \otimes \psi_k \ dm =  \sum\n_{i,j\le d_k} e_{ij} \otimes e_{ij}.$$
Thus if we let $\varphi_k=\theta_k^*$,
then $(\varphi_k,\psi_k)$ satisfies
\eqref{37c} and $(\varphi_k)$ (as well as $(\theta_k)$)
is $\|T\|$-dominated by $(g_k)$.
Therefore, by Theorem \ref{t2} we conclude that
$(\psi_k)$ is ${\dot\otimes}^2$-Sidon.
\end{proof}
 
The next statement records a simple observation.
\begin{pro}\label{Pr} Let $(\varphi_n, \psi_n) $ be systems satisfying
\eqref{33} and \eqref{37}.
In addition
assume that $(\varphi_n)$ satisfies
\eqref{31}.
The following are equivalent:
 \begin{itemize}
\item[(i)] $(  \psi_n) $ is randomly Sidon.
\item[(ii)] $(  \psi_n) $ is randomly ${\dot\otimes}^k$-Sidon for all $k\ge 1$.
\end{itemize}
\end{pro}
\begin{proof} Assume (i). Let $(\psi^1_n)$ be any randomly Sidon system.
We will show that $(\psi^1_n\dot\otimes\psi_n)$ is randomly Sidon.
 Fix $t_2$. Then
 $$\sum d_k \tr|\psi_k(t_2) x_k  |\le C \E {\rm ess}\sup\n_{t_1}|\sum d_k \tr( \psi_k(t_2) x_k g_k\psi^1_k(t_1))|$$ and hence
$$\int \sum d_k \tr|\psi_k  x_k  | dm  \le  C \E {\rm ess}\sup\n_{t_1,t_2}|\sum d_k \tr( \psi_k(t_2) x_k g_k\psi^1_k(t_1))|=C\E \| \sum d_k \tr(x_k g_k\psi^1_k\dot\otimes  \psi_k)\|.$$
Now we claim that \eqref{31} and \eqref{37}  imply
$  \sum d_k \tr|x_k| \le C'\int \sum d_k  \tr|\psi_k(t) x_k  | dm(t) . $
Indeed, let $x_k=v_k|x_k| $ be the polar decomposition.
Then $$\tr|x_k|= \int \tr  ( v_k^* {\varphi_k(t)}^*\psi_k(t) v_k |x_k|) dm(t)
= \int \tr  ( v_k^* {\varphi_k(t)}^*\psi_k(t) x_k) dm(t)\le C' \int   \tr|\psi_k(t) x_k  | dm(t) ,$$ from which the claim  follows. This shows that (i) implies
that $(\psi^1_n\dot\otimes\psi_n)$ is randomly Sidon.
In particular taking $\psi^1_n=\psi_n$ we find that
$(  \psi_n\dot\otimes  \psi_n) $ is randomly  Sidon.
Iterating this   argument  
we obtain (ii).
 (ii) $\Rightarrow$ (i) is trivial. 
\end{proof}
\begin{rem}\label{rbath} By the same reasoning, assuming that $(\varphi_n, \psi_n) $   satisfy 
\eqref{33} and \eqref{37}, and
  that $(\varphi_n)$ satisfies
\eqref{31}, one shows
that 
if $(  \psi_n) $ is   ${\dot\otimes}^k$-Sidon
then it is   ${\dot\otimes}^{k+1}$-Sidon.
\end{rem}
We now come to the main point: the comparison between
Sidon and randomly Sidon. The generalization of Rider's result 
from \cite{Ri}
in our new framework is:  
\begin{thm}\label{rs4} Let $(\varphi_n, \psi_n) $ be systems satisfying
\eqref{31}, \eqref{33} and \eqref{37}. \\
 If  $(\psi_n)$ is randomly Sidon, then
 it  is ${\dot\otimes}^k$-Sidon for all $k\ge 4$.
\end{thm}
\begin{proof}  Proposition \ref{P} with
$\varphi^1=\varphi^2$ and $\psi^1=\psi^2=\psi$
  shows that $(\psi_n {\dot\otimes}^2 \psi_n)$ is ${\dot\otimes}^2$-Sidon.
 Equivalently $(\psi_n)$  is ${\dot\otimes}^4$-Sidon.
 By Remark \ref{rbath}, it is ${\dot\otimes}^k$-Sidon for all $k\ge 4$.
\end{proof}
\begin{cor}[Rider, circa 1975, unpublished]\label{cors4}
A sequence $(\varphi_n)$   of distinct irreducible representations
on a compact group is Sidon iff it is randomly Sidon.\end{cor}
\begin{rem}\label{rbath2}  In the case of function systems, randomly 
${\dot\otimes}^k$-Sidon for some $k$ implies randomly Sidon (see
Remark \ref{rbath0}),
and hence the converse to Theorem \ref{rs4} holds, 
but it seems unclear for general matricial systems. However, the   argument
in Remark \ref{rbath0}
based on Slepian's Lemma does work for randomly \emph{central} ${\dot\otimes}^k$-Sidon.
\end{rem}
 \rem[Rider's unpublished results]
  For subsets of duals of compact non-Abelian groups,
  Rider \cite{Ri} announced in  1975  that he had solved the
  (then still open)  ``union problem"
  by proving that
   the union of two Sidon sets is Sidon.
   He also extended to the  non-Abelian case
   that randomly Sidon implies Sidon. However,
    he never published
   the proof. Around 1979, since I needed to use his result, I wrote
   to him and he kindly sent me  a long detailed handwritten letter
   describing his proof, based on a delicate estimate
   of the ratio of determinants
   appearing in Weyl's famous character formulae \cite{W}
   for representations of the unitary groups. Unfortunately that letter was lost since then
   and Rider passed away in 2008. 
   \\
   In  Corollary \ref{cors4}  we have obtained
a new proof of Rider's unpublished result that a randomly Sidon 
set $\Lambda$  is   Sidon when $\Lambda$ is a
set of   irreducible representations 
on a compact group.   In particular, since 
randomly Sidon is obviously stable by finite unions,
this is the first {\it published} proof of the stability of Sidon sets under finite unions.
(See however \cite{Wi} for \emph{connected} compact  groups, using
 the structure theory of   Lie groups).
 
   In a sequel to the present paper \cite{Pir}
 we present what is most likely but a reconstruction of Rider's original proof.
 The heart of that proof
   is a  uniform spectral gap estimate for the sequence
of the unitary groups $U(n)$ ($n\ge 1$), which  may be of
independent interest for random matrix  or free probability theory.

\section{Sidon sets of characters on a non-Abelian compact group}\label{ssco}

Although we have nothing new to add to this,
we would like to emphasize here a curious phenomenon
already observed in  \cite{Pi2},   concerning
  Sidon sets that are singletons, i.e. simply formed of 
 a single irreducible representation. 
 Though simple, this is a    nontrivial example,
 because the dimension
$d$ is allowed to tend to $\infty$, while the constants remain fixed.
\begin{thm}[\cite{Pi2}]\label{last} Let $\pi$ be an irreducible representation
of dimension $d$
on a compact group $G$ equipped with its Haar probability $m_G$. Let $\chi(x)=\tr(\pi(x))$
be its character.
Assume that  for some constant $C$
\begin{equation}\label{g0}\|\chi\|_{\psi_2}\le C.\end{equation}
Then there is $\alpha$ depending only on $C$ (and \emph{not}
on $d=\dim(\pi)$) such that for  any $a\in M_d$ we have
\begin{equation}\label{g1}\tr|a|\le \alpha \sup_{t\in G} |\tr(a\pi(t))|
.\end{equation}
Conversely, \eqref{g1} for some $\alpha$ implies \eqref{g0} with a constant
$C$ depending only on $\alpha$.
\end{thm}
It seems curious that the subGaussian nature of the character
of an irreducible representation  $\pi$ expressed by \eqref{g0}
suffices by itself to imply the  strong
property of the whole range of $\pi$ expressed by \eqref{g1}.
\begin{proof} Assume \eqref{g0}. 
Let $g$ be a random $d\times d$ Gaussian matrix
as in \S \ref{s3}. Let $\d$ be the metric defined on $G$
by $\d(x,y)= (d \ \tr |\pi(x)-\pi(y)|^2)^{1/2}$.
Let $N(\vp)$ be the smallest number of a covering
of $G$ by open balls of radius $\vp$ for the metric $\d$.
Since this metric is   translation invariant, we have
$$ m_G(\{x\mid \d(x,1)<\vp\})^{-1}\le N(\vp)$$
(and in fact
$  N(\vp)$ is essentially equivalent to
$  m_G(\{x\mid \d(x,1)<\vp\})^{-1}$). Since $\tr |\pi(t)-I|^2=2(d-\Re (\chi(t)))$, we have
  for any $0<\vp<\sqrt 2$
\begin{equation}\label{43} m(\{t\mid \d(t,1)<\vp d \})=m(\{ \Re( \chi) >d(1-\vp^2/2)\}). \le e \exp{-(d^2(1-\vp^2/2)^2C^{-2})}.\end{equation}
 Therefore Sudakov's minoration (see e.g. \cite[p.69]{Piv} or \cite{Ta2})
tells us that that there is a numerical constant $c'$
 such that
\begin{equation}\label{44} \sup_{\vp>0} \vp d \left(\log \frac{1}{m_G(\{t\mid \d(t,1)<\vp d\}}\right)^{1/2} \le c' \E \sup_{t\in G} | d\  \tr(g\pi(t)  |.\end{equation}
Choosing e.g. $\vp=1$, \eqref{43} gives us $d^2(2C)^{-1}-1 \le c'\E \sup_{t\in G} | d\  \tr(g\pi(t)  |$, and hence for all large enough $d $  we have
$d^2(2C)^{-1}/2 \le c'\E \sup_{t\in G} | d\  \tr(g\pi(t)  |$. The latter means
that the singleton $\{\pi\}$ is randomly \emph{central} Sidon with constant 
$c'4C^2/2$. By  Proposition \ref{46}, $\{\pi\dot\otimes \pi\}$, and hence $\{\pi\}$ itself, is Sidon with the same constant.
The converse implication follows from the non-Abelian analogue 
(due to  
   Fig\`a-Talamanca and   Rider) of Rudin's classical
result that Sidon sets have
$\Lambda(p)$-constants growing like $\sqrt p$, and hence 
(recall \eqref{111}) are subGaussian. For details see \cite[(37.25) p. 437]{HR} or \cite{MaPi}.
\end{proof}
In our follow-up paper \cite{Pir} we
review (and partly amend) the results of \cite{Pi2}.
  We refer the reader to \cite{Pir} for more   on
the   themes of the present paper.

\bigskip

\bigskip

 {\it Acknowledgement} The author is grateful to Prof. G. Misra for his invitation
 to visit
 the Math. Dept. of  IISc Bangalore
where part of 
 this paper was written.

 \end{document}